\documentclass[11pt]{article}
\usepackage[T1]{fontenc}
\usepackage{graphicx,amsthm,amsfonts,amssymb,microtype,thmtools,mathtools,tikz}
\usepackage{amsmath}
\usepackage{lmodern}
\usepackage[]{enumitem}
\setlist[itemize]{topsep=0ex,itemsep=0.5ex,parsep=0.4ex}
\setlist[enumerate]{topsep=0ex,itemsep=0.5ex,parsep=0.4ex}
\usepackage{setspace}
\usepackage[unicode=true]{hyperref}
\hypersetup{
    colorlinks=true,
    breaklinks=true,
    linkcolor={blue},
    citecolor={blue},
    urlcolor={brown!60!black},
    pdftitle={Embedding Graphs of Simple Treewidth into Sparse Products},
    pdfauthor={Kevin Hendrey, David~R.~Wood, Jung Hon Yip}}
\usepackage[noabbrev,capitalise]{cleveref}
\crefname{lem}{Lemma}{Lemmas}
\crefname{thm}{Theorem}{Theorems}
\crefname{cor}{Corollary}{Corollaries}
\crefname{prop}{Proposition}{Propositions}
\crefname{conj}{Conjecture}{Conjectures}
\crefname{open}{Open Problem}{Open Problems}
\crefname{claim}{Claim}{Claims}
\crefformat{equation}{(#2#1#3)}
\Crefformat{equation}{Equation #2(#1)#3}

\newcommand{\defn}[1]{\textcolor{green!50!black}{\textit{#1}}}
\newcommand{\mathdefn}[1]{\textcolor{green!50!black}{#1}}
\newcommand{\indeg}{\Delta^-}
\DeclareMathOperator{\att}{att}

\usepackage[longnamesfirst,numbers,sort&compress]{natbib}
\makeatletter
\def\NAT@spacechar{~}
\makeatother
\usepackage[margin=3cm]{geometry}
\usepackage{setspace}
\DeclarePairedDelimiter{\ceil}{\lceil}{\rceil}
\renewcommand{\geq}{\geqslant}
\renewcommand{\leq}{\leqslant}
\newcommand{\StrongProd}{\mathbin{\boxtimes}}
\newlength\squareheight
\newcommand\squareslash{\tikz{\draw (0,0) rectangle (\squareheight,\squareheight);\draw(0,0) -- (\squareheight,\squareheight)}}

\DeclareMathOperator{\stw}{stw}
\DeclareMathOperator{\tw}{tw}

\renewcommand{\thefootnote}{\fnsymbol{footnote}}
\theoremstyle{plain}
\newtheorem{thm}[equation]{Theorem}
\newtheorem{lem}[equation]{Lemma}
\newtheorem{cor}[equation]{Corollary}
\newtheorem{prop}[equation]{Proposition}
\newtheorem{obs}[equation]{Observation}

\newcounter{claimnumber}[equation]

\newtheorem{claim}{Claim}[claimnumber]
\theoremstyle{definition}
\newtheorem{conj}[thm]{Conjecture}
\newcommand{\NN}{\mathbb{N}}

\newcommand{\depth}{\text{depth}}
\newcommand{\score}{\text{score}}
\usepackage[bottom]{footmisc}
\usepackage{float}
\usepackage{thm-restate}
\newcommand{\subsetsim}{\mathrel{\substack{\textstyle\subset\\[-0.3ex]\textstyle\sim}}}
\makeatletter
\newcommand{\customlabel}[2]{%
    \protected@write \@auxout {}{\string \newlabel {#1}{{#2}{}}}}
\makeatother
\usepackage{booktabs}
\usepackage{caption}
\newcommand{\specialcell}[2][c]{%
    \begin{tabular}[#1]{@{}c@{}}#2\end{tabular}}
\usepackage{subcaption}
\begin{document}

\author{
    Kevin Hendrey\footnotemark[2] \qquad
    David~R.~Wood\footnotemark[2] \qquad
    Jung Hon Yip\footnotemark[2]}

\footnotetext[2]{School of Mathematics, Monash   University, Melbourne, Australia, \texttt{\{Kevin.Hendrey1,David.Wood,\linebreak Junghon.Yip\}@monash.edu}. Research of Hendrey is supported by the Australian Research Council. Research of Wood is supported by the Australian Research Council and NSERC. Yip is supported by a Monash Graduate Scholarship.}

% \sloppy

\title{\bf\boldmath Embedding Graphs of Simple Treewidth \\into Sparse Products}

\maketitle

\begin{abstract}
    We study embeddings of graphs with bounded treewidth or bounded simple treewidth into the undirected graph underlying the directed product of two directed graphs. If the factors have bounded maximum indegrees, then the product graph has bounded maximum indegree and therefore is sparse. We prove that every graph of simple treewidth $k$ is contained in (ignoring edge directions) the directed product of directed graphs $\vec H_1$ and $\vec H_2$, with $\indeg(\vec H_1), \indeg(\vec H_2) \leq k-1 $ and $\tw(H_1), \tw(H_2) \leq k-1$.  Further, we show that this treewidth bound is best possible. Several corollaries follow from our results: every outerplanar graph is contained in the directed product of trees with maximum indegree $1$, and every planar graph with treewidth $3$ is contained in a directed product of graphs with treewidth $2$ and maximum indegree $2$. However, for graphs of treewidth $k$, we prove a negative result: for any integers $s, t, k \geq 1$, there is a graph $G$ with treewidth $k$ not contained in the directed product of $\vec H_1$ and $\vec H_2$ for any directed graphs $\vec H_1$ and $\vec H_2$ with $\indeg(\vec H_1) \leq s$, $\indeg(\vec H_2) \leq t$, and $\tw(H_1), \tw(H_2) \leq k-1$. This result stands in stark contrast to the strong product case, where Liu, Norin and Wood [arXiv:2410.20333] proved that the optimal lower bound on the factors is about half the treewidth.
\end{abstract}

\renewcommand{\thefootnote}
{\arabic{footnote}}

%\newpage
% \tableofcontents

\section{\boldmath Introduction}
\label{s:intro}
Graph product structure theory describes graphs in complicated graph classes as subgraphs of products of graphs in simpler graph classes, typically with bounded treewidth or bounded pathwidth. As defined in \cref{ss:treewidth}, the treewidth of a graph $G$, denoted by \defn{$\tw(G)$}, is the standard measure of how similar $G$ is to a tree\footnote{Unless stated otherwise, all graphs are simple and finite.}.

As illustrated in \cref{fig:directed_grid}, the \defn{strong product $H_1 \StrongProd H_2$} of graphs $H_1$ and $H_2$ has vertex set $V(H_1) \times V(H_2)$,
where distinct vertices $(x,y), (x',y')$ are adjacent if $x = x'$ and $yy' \in E(H_2)$, or $y = y'$ and $xx' \in E(H_1)$, or $xx'\in E(H_1)$ and $yy' \in E(H_2)$. Edges of the last type are called \defn{diagonal} edges.

The following theorem of
\citet*{DJMMUW20} is the classical example of a graph product structure theorem. A graph $H$ is \defn{contained} in a graph $G$ if $H$ is isomorphic to a subgraph of $G$, denoted \defn{$H \subsetsim G$}. The complete graph on $n$ vertices is denoted \defn{$K_n$}.

\begin{thm}[Planar Graph Product Structure Theorem \cite{DJMMUW20}]
    \label{thm:PGPST}
    Every planar graph $G$ is contained in $H \boxtimes P \boxtimes K_3$ for some planar graph $H$ with $\tw(H) \leq 3$ and path $P$.
\end{thm}

\cref{thm:PGPST} provides a powerful tool for studying questions about planar graphs, by reducing to graphs of bounded treewidth. Indeed, \cref{thm:PGPST} has been the key to resolving several open problems regarding queue layouts~\cite{DJMMUW20}, nonrepetitive colourings~\cite{DEJWW20}, centred colourings~\cite{DFMS21}, adjacency ordering schemes~\cite{GJ22,BGP22,EJM23,DEGJMM21}, twin-width~\cite{BKW,KPS24,JP22}, infinite graphs~\cite{HMSTW}, and comparable box dimension~\cite{DGLTU22}.

In \cref{thm:PGPST}, the treewidth $3$ bound is best possible (see Theorem 19 in \cite{DJMMUW20}). However, \citet{distel2024treewidth} proved that treewidth $2$ is achievable if one does not require the second factor to be a path.
\begin{thm}[\cite{distel2024treewidth}]
    \label{thm:Planar222}
    Every planar graph $G$ is contained in $H_1 \StrongProd H_2 \StrongProd K_2$ for some graphs $H_1$ and $H_2$ with $\tw(H_1)\leq 2$ and $\tw(H_2)\leq 2$.
\end{thm}
\cref{thm:Planar222} is best possible in the sense that for any $c \in \NN$, there is a planar graph $G$ such that for any tree $T$ and graph $H$ with $\tw(H) \leq 2$, $G$ is not contained in $H \boxtimes T \boxtimes K_c$ \citep{distel2024treewidth}.

Note that $H_1\StrongProd H_2$ may be dense even if both $H_1$ and $H_2$ have bounded treewidth. For example, the strong product $K_{1,n}\StrongProd K_{1,m}$ contains the complete bipartite graph $K_{n,m}$. Therefore, a weak point of \cref{thm:Planar222} is that it embeds planar graphs (which are sparse) in products of bounded treewidth graphs (which may be dense). In this paper we consider the directed product as a way to address this issue, building on previous work of \citet{LNW}.

\subsection{Directed Product}
\label{ss:directed_prod}
We use $G$ to represent an undirected graph, and $\vec G$ to represent a directed graph (or \defn{digraph} for short).
To avoid confusion, the term \defn{arc} refers to directed edges in digraphs, while the term \defn{edge} refers to undirected edges in undirected graphs.
We use the notation \defn{$\overrightarrow{uv}$} to denote an arc in $E(\vec G)$ directed from $u$ to $v$. Otherwise, $uv$ represents an undirected edge in $E(G)$.
A digraph $\vec G$ is \defn{oriented} if $\overrightarrow{uv} \in E(\vec G)$ implies $\overrightarrow{vu} \notin E(\vec G)$.
Given a digraph $\vec G$, we use $G$ to represent the undirected graph on the same vertex set as $\vec G$, with $uv \in E(G)$ if and only if $\overrightarrow{uv} \in E(\vec G)$ or $\overrightarrow{vu} \in E(\vec G)$. We call $G$ the \defn{underlying graph} of $\vec G$.

Let $\vec H_1$ and $\vec H_2$ be digraphs. As shown in \cref{fig:directed_grid}, the \defn{directed product $\overrightarrow{\vec H_1 \, \squareslash \, \vec H_2}$} of digraphs $\vec H_1$ and $\vec H_2$ is the directed graph with vertex set $V(H_1)\times V(H_2)$, where $\overrightarrow{(x,y)(x',y')} \in E(\overrightarrow{\vec H_1 \, \squareslash \, \vec H_2})$ is an arc if $x=x'$ and $\overrightarrow{yy'}\in E(\vec H_2)$, or $y=y'$ and $\overrightarrow{xx'} \in E(\vec H_1)$, or $\overrightarrow{xx'} \in E(\vec H_1)$ and $\overrightarrow{yy'} \in E(\vec H_2)$. Arcs of the last type are called \defn{diagonal} arcs.

\begin{figure}[t]
    \centering
    \includegraphics{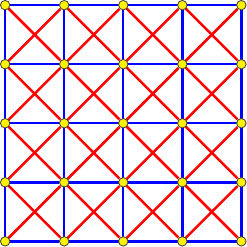}
    \qquad
    \includegraphics{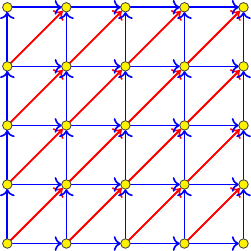}
    \qquad
    \includegraphics{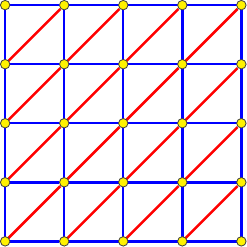}
    \caption{Left: The strong product of two paths. Center: The directed product of two directed paths. Right: The underlying graph of the directed product of two directed paths. The red edges are the diagonal edges.}
    \label{fig:directed_grid}
\end{figure}
The directed product is the natural generalisation of the strong product to directed graphs, and appears in several places in the literature \citep{LNW, bang2018classes}.

The digraphs $\vec H_1$ and $\vec H_2$ are called the \defn{factors} of $\overrightarrow{\vec H_1 \, \squareslash \, \vec H_2}$. An edge of $E(\vec H_1 \, \squareslash \, \vec H_2)$ is \defn{diagonal} if it arises from a diagonal arc in $\overrightarrow{\vec H_1 \, \squareslash \, \vec H_2}$. This paper is mainly concerned with the underlying undirected graph of $\overrightarrow{\vec H_1 \, \squareslash \, \vec H_2}$, denoted by \defn{$\vec H_1 \, \squareslash \, \vec H_2$}. See \cref{fig:directed_grid} for details.

For a digraph $\vec H$ and $v \in V(H)$, let \defn{$d_{\vec H}^-(v)$} be the \defn{indegree} of $v$ in $\vec H$. Let \defn{$\indeg(\vec H)$} be the maximum indegree of $\vec H$. For any orientation of the edges of the graphs $H_1$ and $H_2$, we have $\vec H_1 \, \squareslash \, \vec H_2 \subseteq H_1 \boxtimes H_2$.

The key benefit in considering the directed product over the strong product is that if the factors have bounded maximum indegree, then so does the product. More precisely,
consider digraphs $\vec H_1$ and $\vec H_2$ with  $\indeg(\vec H_1) \leq s$ and $\indeg(\vec H_2) \leq t$. Then $\indeg(\overrightarrow{\vec H_1 \, \squareslash \, \vec H_2}) \leq st + s + t$. Let $H'$ be a subgraph of $\vec H_1 \, \squareslash \, \vec H_2$, and let $\vec H$ be the subgraph of $\overrightarrow{\vec H_1 \, \squareslash \, \vec H_2}$ induced by $V(H')$. Then,
\begin{equation*}
    |E(H')| \leq  |E(\vec H)| = \sum_{v \in V(H)} d_{\vec H}^{-}(v) \leq |V(H')| \indeg(\vec H) \leq (st + s + t) |V(H')|,
\end{equation*}
and so $H'$ has bounded average degree. As an illustrative example, the directed product of two large stars $K_{1,n}$ and $K_{1,m}$ does not contain $K_{m,n}$ if the arcs of each factor are oriented away from the centre, unlike in the strong product case. Therefore, showing that a graph $G$ is contained in the underlying graph of a directed product of two digraphs with
bounded maximum indegree  is qualitatively stronger than showing that $G$ is contained in the strong product of the corresponding underlying undirected graphs.

Our first contributions are the following results, which are qualitative improvements to \cref{thm:Planar222} for two well-studied special classes of planar graphs.

\begin{restatable}{thm}{outerplanar}
    \label{thm:outerplanar}
    Every outerplanar graph is contained in $\vec T_1 \, \squareslash \, \vec T_2$ for some oriented trees $\vec T_1$ and $\vec T_2$ with $\indeg(\vec T_1), \indeg(\vec T_2) \leq 1$.
\end{restatable}

\begin{thm}
    \label{thm:apollonian}
    Every planar graph with treewidth $3$ is contained in $\vec H_1 \, \squareslash \, \vec H_2$ for some oriented digraphs $\vec H_1$ and $\vec H_2$ with $\tw(H_1), \tw(H_2) \leq 2$, and $\indeg(\vec H_1), \indeg(\vec H_2) \leq 2$.
\end{thm}

We show that the treewidth bound on the factors in \cref{thm:apollonian} is best possible. In particular, \cref{thm:sp_lb_stw} implies that there is a planar graph with treewidth $3$ not contained in the strong product of two trees. In \cref{thm:dp_ub_stw}, we actually prove a substantial generalisation of \cref{thm:outerplanar,thm:apollonian} in terms of `simple treewidth', which we introduce below.

\subsection{Treewidth and Simple Treewidth}
\label{ss:treewidth}
For a non-null tree $T$, a \defn{$T$-decomposition} of a graph $G$ is a collection $\mathcal{B}= (B_x:x \in V(T))$ of subsets of $V(G)$ such that:

\begin{itemize}
    \item for each edge ${vw \in E(G)}$, there exists a node ${x \in V(T)}$ with ${v,w \in B_x}$, and
    \item for each vertex ${v \in V(G)}$, the set $\{ x \in V(T) : v \in B_x \}$ induces a non-empty (connected) subtree of $T$.
\end{itemize}
The \defn{width} of such a $T$-decomposition is ${\max\{ |B_x| : x \in V(T) \}-1}$. A \defn{tree-decomposition} is an ordered pair $(T,\mathcal{B})$ consisting of a tree $T$ and a $T$-decomposition $\mathcal{B}$. The \defn{treewidth} of a graph $G$, denoted \defn{$\tw(G)$}, is the minimum width of a tree-decomposition of $G$. Treewidth is the standard measure of how similar a graph is to a tree. For example, a connected graph has treewidth at most 1 if and only if it is a tree. It is an important parameter in structural graph theory, especially in Robertson and Seymour's graph minor theory, and also in algorithmic graph theory, since many NP-complete problems are solvable in linear time on graphs with bounded treewidth. See \citep{HW17,Bodlaender98,Reed97} for surveys on treewidth.

A tree-decomposition $(T, \mathcal{B})$ of a graph $G$ is \defn{$k$-simple}, for some $k\in\NN$,  if it has  width  at most $k$, and for every set $S$ of $k$ vertices in $G$, we have $|\{x\in V(T): S\subseteq B_x\}|\leq 2$. The \defn{simple treewidth} of a graph $G$, denoted by \defn{$\stw(G)$}, is the minimum $k\in\NN$ such that $G$ has a $k$-simple tree-decomposition. Simple treewidth appears in several places in the literature under various guises \citep{KU12,KV12,MJP06,Wulf16}. The following facts are well known: A connected graph has simple treewidth 1 if and only if it is a path. A graph has simple treewidth at most 2 if and only if it is outerplanar. A graph has simple treewidth at most 3 if and only if it has treewidth 3 and is planar~\citep{KV12}. The edge-maximal  graphs with simple treewidth 3 are ubiquitous objects, called  \defn{planar 3-trees} in structural graph theory and graph drawing~\citep{AP-SJADM96,KV12}, \defn{stacked polytopes} in polytope theory~\citep{Chen16}, and \defn{Apollonian networks} in enumerative and random graph theory~\citep{FT14}. It is also known and easily proved that $\tw(G) - 1 \leq \stw(G)\leq \tw(G)$ for every graph $G$ (see \citep{KU12,Wulf16}).

We remark that there is a notion of `directed treewidth' defined by \citet{JRST-JCTB01}. However, this paper is only concerned with the treewidth of undirected graphs.

\textbf{Goals of the paper.} This paper studies non-trivial embeddings of graphs with treewidth $k$ and graphs with simple treewidth $k$ in three settings: strong products, directed products with unbounded indegree, and directed products with bounded indegree. By `non-trivial', we want the treewidth of the factors to be strictly less than $k$, and the goal is to minimise the treewidth of the factors. The strong product setting was studied by \citet*{LNW}, who obtained tight bounds on the treewidth of the factors when embedding treewidth $k$ graphs. We extend their results to simple treewidth $k$ graphs, and give tight upper and lower bounds on the treewidth of the factors in the other two settings. All of these bounds are summarised in \cref{tab:summary}. Our main contributions are \cref{thm:dp_lb,thm:dp_ub_stw,thm:sp_lb_stw}.
\begin{table}[H]
    % \captionsetup{justification=left}s
    \centering
    \caption[]{The upper bound in each setting means that every graph in $\mathcal{G}$ is contained in $H_1 \boxtimes H_2$ (or $\vec H_1 \, \squareslash \, \vec H_2$) for some graphs $H_1$ and $H_2$ with $p:=\tw(H_1)$ and $q:=\tw(H_2)$ satisfying the upper bound. The lower bound in each setting means that there is a graph $G \in \mathcal{G}$ such that if $G \subsetsim H_1 \boxtimes H_2$ (or $G \subsetsim \vec H_1 \, \squareslash \, \vec H_2$), then $p:=\tw(H_1)$ and $q:=\tw(H_2)$ satisfy the lower bound.}
    \begin{tabular}{|p{12mm}|p{28mm}|p{28mm}|p{28mm}|}
        \hline
        \multicolumn{1}{|c|}{\specialcell[t]{\textbf{graph class} $\mathcal{G}$}} & \multicolumn{1}{c|}{\specialcell[t]{\textbf{strong product}}}             & \multicolumn{1}{c|}{\specialcell[t]{\textbf{directed product} \\ \textbf{$\indeg(H_i)$ unbounded}}} & \multicolumn{1}{c|}{\specialcell[t]{\textbf{directed product} \\ \textbf{$\indeg(H_i)$ bounded}}} \\
        \hline
        \multicolumn{1}{|c|}{\specialcell[t]{$\tw(G) \leq k$}}                    & \multicolumn{2}{c|}{\specialcell[t]{$p + q \leq k +1
        $ (\ref{prop:sp_ub}, \ref{lem:dpui_ub})                                                                                                                                                                               \\$p + q \geq k +1$ (\ref{lem:sp_lb}, \ref{cor:dpui_lb}) } } & \multicolumn{1}{c|}{\specialcell[t]{\\$p \geq k$ or $q \geq k$ (\ref{thm:dp_lb})}}  \\
        \hline
        \multicolumn{1}{|c|}{\specialcell[t]{$\stw(G) \leq k$}}                   & \multicolumn{2}{c|}{\specialcell[t]{$p + q \leq k +1$  (\ref{prop:sp_ub})                                                                 \\$p + q \geq k +1$ for $p,q \leq k-2$ (\ref{thm:sp_lb_stw},
        \ref{cor:dpui_lb_stw})}}                                                  & \multicolumn{1}{c|}{\specialcell[t]{$p,q \leq k-1$ (\ref{thm:dp_ub_stw})                                                                  \\$p \geq k-1$ or $q \geq k-1$ (\ref{cor:dp_lb_stw})}}
        \\
        \hline
    \end{tabular}
    \label{tab:summary}
\end{table}

\subsection{Embedding into Strong Products}
\label{ss:strong_product}
\citet*[Proposition 3.6]{LNW} showed\footnote{\citet*{LNW} actually showed that every graph of treewidth $k$ is contained in $H_1 \boxtimes H_2$ for some graphs $H_1$ and $H_2$ with $\tw(H_i) \leq \ceil{(k+1)/2}$, but their proof can be modified to prove \cref{prop:sp_ub}.}:
\begin{prop}
    \label{prop:sp_ub}
    For any integers $p, q, k \geq 1$ with $p + q \geq k+1$, every graph $G$ of treewidth $k$ is contained in $H_1 \boxtimes H_2$ for some graphs $H_1$ and $H_2$ with $\tw(H_1) \leq p$ and $\tw(H_2) \leq q$.
\end{prop}

\citet*[Lemma 4.6]{LNW} also proved a tight lower bound:
\begin{lem}
    \label{lem:sp_lb}
    For any integers $p, q, k \geq 1$ with $p + q \leq k$, there is a graph $G$ with treewidth $k$ that is not contained in $H_1 \boxtimes H_2$ for any graphs $H_1$ and $H_2$ with $\tw(H_1) \leq p$ and $\tw(H_2) \leq q$.
\end{lem}
In the case $p, q \leq k-2$, we prove a strengthening of \cref{lem:sp_lb}:

\begin{restatable}{thm}{restatablesplbstw}
    \label{thm:sp_lb_stw}
    For any integers $p, q \geq 1$ with $p, q \leq k - 2$ and $p + q \leq k$, there is a graph with simple treewidth $k$ that is not contained in $H_1 \boxtimes H_2$ for any graphs $H_1$ and $H_2$ with $\tw(H_1) \leq p$ and $\tw(H_2) \leq q$.
\end{restatable}
Note that the $p, q \leq k -2$ assumption in \cref{thm:sp_lb_stw} implies $k \geq 3$. Indeed, the conclusion of \cref{thm:sp_lb_stw} is false with $k = 2$, and $p,q=1$. In particular, outerplanar graphs have simple treewidth $2$, and are contained in the strong product of two trees, as stated in \cref{thm:outerplanar}.

\subsection{Embedding into Directed Products with Unbounded Indegree}
\label{ss:dpui}
As explained in \cref{ss:directed_prod}, if the indegrees of the factors are unbounded, then the directed product is not sparse, and behaves more like the strong product. The method used to prove \cref{prop:sp_ub} can be used to show the following strengthening:
\begin{restatable}{lem}{restatabledpuiub}
    \label{lem:dpui_ub}
    For any integers $p, q, k \geq 1$ with $p + q \geq k + 1$, every graph of treewidth $k$ is contained in $\vec H_1 \, \squareslash \, \vec H_2$ for some oriented digraphs $\vec H_1$ and $\vec H_2$ with $\tw(H_1) \leq p$ and $\tw(H_2) \leq q$ respectively.
\end{restatable}
For completeness, we provide the proof of \cref{lem:dpui_ub} in \ref{appendix:A}. Since $\vec H_1 \, \squareslash \, \vec H_2 \subseteq H_1 \boxtimes H_2$ for any digraphs $\vec H_1$ and $\vec H_2$, \cref{prop:sp_ub} follows from \cref{lem:dpui_ub}, and the following corollaries are immediate from \cref{lem:sp_lb,thm:sp_lb_stw}.

\begin{cor}
    \label{cor:dpui_lb}
    For any integers $p, q, k \geq 1$ with $p + q \leq k$, there is a graph with treewidth $k$ that is not contained in $\vec H_1 \, \squareslash \, \vec H_2$ for any graphs $\vec H_1$ and $\vec H_2$ with $\tw(H_1) \leq p$ and $\tw(H_2) \leq q$.
\end{cor}

\begin{cor}
    \label{cor:dpui_lb_stw}
    For any integers $p,q,k \geq 1$ with $p, q \leq k-2$ and $p + q \leq k$, there is a graph with simple treewidth $k$ that is not contained in $\vec H_1 \, \squareslash \, \vec H_2$ for any digraphs $\vec H_1$ and $\vec H_2$ with $\tw(H_1) \leq p$ and $\tw(H_2) \leq q$.
\end{cor}

\subsection{Embedding into Directed Products with Bounded Indegree}
\label{ss:dp}
The main contributions of this paper are the following two theorems. The first theorem says that it is impossible to embed every graph of treewidth $k$ in a directed product of digraphs with bounded maximum indegree whose underlying graphs have strictly smaller treewidth.

\begin{thm}
    \label{thm:dp_lb}
    For any integers $s, t, k \geq 1$, there is a graph of treewidth $k$ that is not contained in $\vec H_1 \, \squareslash \, \vec H_2$ for any digraphs $\vec H_1$ and $\vec H_2$ with $\indeg(\vec H_1) \leq s, \indeg(\vec H_2) \leq t$ and $\tw(H_1), \tw(H_2) \leq k -1$.
\end{thm}
The main obstacle for embedding treewidth $k$ graphs in the directed product of factors with strictly smaller treewidth is the fact that treewidth $k$ graphs can have an arbitrary number of vertices adjacent to all vertices of a $k$-clique. Simple treewidth $k$ graphs avoid this problem, and we get the following positive result:
\begin{restatable}{thm}{restatedpubstw}
    \label{thm:dp_ub_stw}
    For each integer $k\geq 1$, every graph of simple treewidth $k$ is contained in $\vec H_1 \, \squareslash \, \vec H_2$ for some oriented digraphs $\vec H_1$ and $\vec H_2$ with $\tw(H_1), \tw(H_2) \leq k-1$, and $\indeg(\vec H_1), \indeg(\vec H_2) \leq k -1$.
\end{restatable}

Since every graph with treewidth $k$ has simple treewidth at most $k+1$, \cref{thm:dp_lb} implies that the treewidth bounds in \cref{thm:dp_ub_stw} are tight:
\begin{cor}
    \label{cor:dp_lb_stw}
    For any integers $s, t, k \geq 1$, there is a graph of simple treewidth $k$ that is not contained in $\vec H_1 \, \squareslash \, \vec H_2$ for any digraphs $\vec H_1$ and $\vec H_2$ with $\indeg(\vec H_1) \leq s, \indeg(\vec H_2) \leq t$ and $\tw(H_1), \tw(H_2) \leq k -2$.
\end{cor}
\cref{thm:dp_ub_stw} implies \cref{thm:outerplanar,thm:apollonian}. However, we prove \cref{thm:outerplanar} independently to give the reader a better intuition of the induction step used in the proof. \cref{thm:outerplanar} is interesting even in the strong product setting, since it is not implied by \cref{prop:sp_ub}.

\textbf{Outline of the paper.} In \cref{s:preliminaries}, we define notation used throughout the proofs. In \cref{s:stw_graphs}, we prove \cref{thm:outerplanar} and \cref{thm:dp_ub_stw}, which are the positive results for graphs of bounded simple treewidth. Then in \cref{s:optimal_bounds} we show the negative result of \cref{thm:dp_lb}. In \cref{s:stw_strong_products} we prove \cref{thm:sp_lb_stw}.

\section{Preliminaries}
\label{s:preliminaries}
We adopt the following notation for the rest of the paper. Let $G$ be a graph, and let $\vec H_1$ and $\vec H_2$ be digraphs. The statement $G \subsetsim \vec H_1 \, \squareslash \, \vec H_2$ is equivalent to saying that there is an injective map $\alpha: V(G) \to V(H_1) \times V(H_2)$ such that if $g_1 g_2 \in E(G)$, then ${\alpha(g_1) \alpha(g_2) \in E(\vec H_1 \, \squareslash \, \vec H_2)}$.
Whenever $G \subsetsim \vec H_1 \, \squareslash \, \vec H_2$, we will implicitly fix such an \defn{embedding} $\alpha$ and identify a vertex $g \in V(G)$ with the pair $\alpha(g) = (v,w) \in V(H_1) \times V(H_2)$. A vertex $(v,w) \in V(H_1) \times V(H_2)$ is \defn{used} if $\alpha(g) = (v,w)$ for some $g \in V(G)$.
For $x= (u, v) \in V(G)$, we refer to $u$ and $v$ as the \defn{$H_1$-projection} and the \defn{$H_2$-projection} of $x$ respectively. Define $\defn{$P_1(x)$}:=u$ and $\defn{$P_2(x)$}:=v$. For a subset $S \subseteq V(G)$, define $\defn{$P_1(S)$} := \{ P_1(x) : x \in S \} \subseteq V(H_1)$ and $\defn{$P_2(S)$} := \{ P_2(x) : x \in S \} \subseteq V(H_2)$.

A clique $C$ is a \defn{$k$-clique} if $|C| = k$. If $C \subseteq V(G)$ is a clique with $|P_1(C)| = p$ and $|P_2(C)| = q$, then we say $C$ is \defn{$(p,q)$-projected}.
We call a $k$-clique $C \subseteq V(G)$ \defn{diagonal} if it is $(k,k)$-projected. Equivalently, a clique $C$ is diagonal if every edge in $C$ is diagonal.

The above definitions all relate to embeddings in directed products ($G \subsetsim \vec H_1 \, \squareslash \, \vec H_2$), but we use their natural analogues in the strong product case ($G \subsetsim H_1\boxtimes H_2$).

If $G$ is a graph with induced subgraphs $G_1$, $G_2$, and $S$, such that $G = G_1 \cup G_2$ and $S = G_1 \cap G_2$, then we say that $G$ arises by \defn{pasting} $G_1$ and $G_2$ along $S$. For $v \in V(G)$, let \defn{$N_G(v)$} denote the neighbours of $v$ in $G$. If $S$ is a set, then \defn{$S+v$} denotes $S \cup \{ v \}$, and \defn{$S-v$} denotes $S \setminus \{ v \}$. All other notation is standard in graph theory; see \citet{Diestel05}.

For an integer $k\geq 0$, a \defn{$k$-tree} is a graph that can be constructed, starting from $K_{k+1}$, by repeatedly adding a new vertex $v$ so that its neighbourhood is an existing $k$-clique. A \defn{simple $k$-tree} is a graph that can be constructed, starting from $K_{k+1}$, by repeatedly adding a new vertex $v$ so that its neighbourhood is an existing $k$-clique, as long as the same $k$-clique is not used more than once.
\cref{lem:k_tree} is well-known, and characterises edge-maximal treewidth $k$ graphs.
\begin{lem}[Theorem 1, \citep{Bodlaender98}]
    \label{lem:k_tree}
    Every graph with treewidth at most $k$ is a subgraph of a $k$-tree.
\end{lem}
Recall that we define a graph $G$ to have $\stw(G) \leq k$ if $G$ has a $k$-simple tree-decomposition, where a tree-decomposition $(T, \mathcal{B})$ is $k$-simple if it has width at most $k$, and for each set $S \subseteq V(G)$ of $k$ vertices, $|\{x \in V(T): S \subseteq B_x\}| \leq 2$. This definition is due to \citet*{HMSTW}.
Several authors \citep{KU12,KV12,MJP06,Wulf16} instead define a graph $G$ to have $\stw(G) \leq k$ if and only if $G$ is a subgraph of a simple $k$-tree.
It is not straightforward to see that the definitions are equivalent.
We establish the equivalence of the two definitions by proving the following lemma in \ref{appendix:B}.
\begin{restatable}{lem}{restatelemstw}
    \label{lem:edge_max_stw_k}
    Every graph has simple treewidth at most $k$ if and only if it is a subgraph of a simple $k$-tree.
\end{restatable}
Note that \cref{lem:edge_max_stw_k} is false for infinite graphs.
For example, the disjoint union of two infinite $2$-way paths is edge-maximal with simple treewidth $1$, but is not a subgraph of a simple $1$-tree \citep*{HMSTW}\footnote{Under a reasonable extension of the definition of a simple $1$-tree to the infinite case.}.

Let $\mathdefn{K_{\overline{k}, 3}}$ be the graph obtained from the complete bipartite graph $K_{k,3}$ by making all the vertices in the part with $k$ vertices pairwise adjacent. Let $S \subseteq V(K_{\overline{k}, 3})$ be the part with $k$ vertices and let $\{b_1, b_2, b_3\} = V(K_{\overline{k}, 3}) \setminus S$. For each $i \in \{1, 2,3\}$, $(S + b_i)$ is a $(k+1)$-clique. Therefore, for any tree-decomposition of width $k$ of $K_{\overline{k}, 3}$, $S$ is contained in at least three bags. Consequently, if $G$ contains $K_{\overline{k}, 3}$, then $G$ cannot have a $k$-simple tree-decomposition, and $\stw(G) > k$.

\section{Simple Treewidth Graphs}

\label{s:stw_graphs}
In this section, we prove \cref{thm:outerplanar,thm:dp_ub_stw}, which are the positive results for graphs of bounded simple treewidth.

\subsection{Outerplanar Graphs}
\label{ss:outerplanar}
If $\vec H_1$ and $\vec H_2$ are digraphs and $(v_i, w_i)(v_j, w_j) \in E(\vec H_1 \, \squareslash \, \vec H_2)$ is a diagonal edge, then $(v_i, w_j)$ and $(v_j, w_i)$ are the \defn{siblings} of $(v_i, w_i)(v_j, w_j)$. If $G \subsetsim \vec H_1 \, \squareslash \, \vec H_2$, we call the set $\{ (v_i, w_j), (v_j, w_i) : (v_i, w_i)(v_j, w_j) \in E(G) \text{ is a diagonal edge}\} $ the \defn{$G$-siblings}. A sibling is used or unused; recall the definition of used in \cref{s:preliminaries}.

\outerplanar*

\begin{proof}
    If $|V(G)| \leq 2$, then the theorem holds with $\vec T_{1}$ and $\vec T_2$ defined as follows: $V(\vec T_{1}) = \{1,2\}$, $E(\vec T_{1}) = \{\overrightarrow{12}\}$, $V(\vec T_2) = \{ 1 \}$ and $E(\vec T_2) = \varnothing$.
    Let $G$ be an edge-maximal outerplanar graph with at least three vertices. Then $G$ can be constructed recursively from triangles by pasting triangles along edges on the outerface (see \citep[Proposition~8.3.1]{Diestel05}). Let $G_t$ be the outerplanar graph at the $t$-th step of this process. So $G_1$ is a triangle.

    We recursively construct oriented trees $\vec T_{1,t}$ and $\vec T_{2,t}$ such that the following invariants are maintained:
    \begin{enumerate}[label=\textbf{(I\arabic*):},leftmargin=4em]
        \item $G_t \subsetsim \vec T_{1,t} \, \squareslash \, \vec T_{2,t}$.
        \item For each $i \in \{1,2\}$, $\indeg(\vec T_{i, t}) \leq 1$.
        \item For each diagonal edge $(v_i, w_i)(v_j, w_j)$ in $G_t$, each used sibling of $(v_i, w_i)(v_j, w_j)$ is adjacent to $(v_i, w_i)$ and $(v_j, w_j)$ in $G$.
        \item No two diagonal edges in $G_t$ share a common unused sibling.
    \end{enumerate}
    For $t=1$, $G_1$ is a triangle. We define $\vec T_{i,1}$ by $V(\vec T_{i,1}) = \{1, 2\}$ and $E(\vec T_{i,1}) = \{\overrightarrow{12}\}$ for each $i \in \{1, 2\}$. Embed $V(G_1)$ at the coordinates $(0,0)$, $(1,0)$, $(1,1)$. Then $\vec T_{1,1}$ and $\vec T_{2,1}$ satisfy (I2). By definition of the directed product, (I1) holds. The only diagonal edge in $G_1$ is $(0,0)(1,1)$, and the only used sibling of $(0,0)(1,1)$ is $(1,0)$, so (I3) and (I4) hold.

    Suppose by induction that we have constructed $G_t \subsetsim \vec T_{1, t} \, \squareslash \, \vec T_{2, t}$ satisfying (I1)--(I4). Let $g_{t+1}$ be the new vertex in $V(G_{t+1}) \setminus V(G_t)$, and its neighbours are $(v_i, w_i)$ and $(v_j, w_j)$. The edge $(v_i, w_i)(v_j, w_j)$ is a diagonal edge or not a diagonal edge.

    \textbf{Case 1:} Suppose $(v_i, w_i) (v_j, w_j)$ is not diagonal, with $v_i = v_j$ or $w_i = w_j$. Without loss of generality, assume $v_i = v_j$, and set $v:= v_i = v_j$. Then $w_i w_j \in E(T_{2, t})$. Define $\vec T_{2, t+1}:= \vec T_{2, t}$ and $\vec T_{1, t+1}$ to be the directed tree obtained from $\vec T_{1, t}$ by adding a new vertex $x$ as a leaf adjacent to $v$, and directing the arc $vx$ away from $v$ and towards $x$. Then (I2) holds.
    Without loss of generality, assume $\overrightarrow{w_i w_j} \in E(\vec T_{2, t})$. Embed $g_{t+1}$ at $(x, w_j)$. To verify (I1), we need to check that $(v, w_j) (x, w_j)$ and $(v, w_i) (x, w_j)$ are embedded. $(v, w_j) (x, w_j)$ is embedded because $vx \in E(T_{1, t+1})$, whereas $(v, w_i) (x, w_j)$ is embedded as a diagonal edge.

    We check (I3) for $G_{t+1}$. Note that $(x, w_j)$ is not a sibling of any existing diagonal edge in $G_t$ since $x$ is a new vertex, so (I3) holds for each existing diagonal edge in $G_t$. The only diagonal edge in $G_{t+1}$ not in $G_t$ is $(v, w_i) (x, w_j)$. Since $x$ is a new vertex, the only used sibling of this edge is $(v, w_j)$, which is adjacent to $(v, w_i)$ and $(x, w_j)$ in $G_{t+1}$ by construction. Lastly, (I4) holds because the only unused $G_{t+1}$-sibling that is not a $G_t$-sibling is $(x,w_i)$, and $(x,w_i)$ is not a sibling for any existing diagonal edge in $G_t$ since $x$ is a new vertex.

    \textbf{Case 2:} Suppose $(v_i, w_i)(v_j, w_j)$ is diagonal, with $\overrightarrow{v_iv_j} \in E(\vec T_{1, t})$ and $\overrightarrow{w_i w_j} \in E(\vec T_{2, t})$. We claim that $(v_i, w_j)$ or $(v_j, w_i)$ is not used in $G_t$. Otherwise, by (I3), $(v_i, w_j)$ and $(v_j, w_i)$ are adjacent to both $(v_i, w_i)$ and $(v_j, w_j)$ in $G_t$. Since $g_{t+1}$ is adjacent to $(v_i, w_i)$ and $(v_j, w_j)$ as well, $K_{\overline{2},3}$ is a subgraph of $G_{t+1}$, contradicting the outerplanarity of $G_{t+1}$. Without loss of generality, assume $(v_i, w_j)$ is not used in $G_t$. Embed $g_{t+1}$ at $(v_i, w_j)$ and define $\vec T_{1, t+1}:= \vec T_{1, t}$ and $\vec T_{2, t+1}:= \vec T_{2, t}$. Then (I2) holds. To show (I1), we need to check that $(v_i, w_j)(v_i, w_i)$ and $(v_i, w_j)(v_j, w_j)$ are embedded. $(v_i, w_j)(v_i, w_i)$ is embedded because $w_iw_j \in E(T_{2, t})$, and $(v_i, w_j)(v_j, w_j)$ is embedded because $v_i v_j \in E(T_{1,t})$.

    We check (I3) for $G_{t+1}$. All diagonal edges in $G_{t+1}$ are in $G_t$, so (I3) holds if we show that apart from $(v_i, w_j)$, an unused $G_t$-sibling remains unused in $G_{t+1}$. The only used $G_{t+1}$-sibling that is not a $G_t$-sibling is $(v_i, w_j)$, and it is an unused sibling of $(v_i, w_i)(v_j, w_j)$ in $G_t$. By (I4) for $G_t$, no other diagonal edge in $G_t$ has $(v_i, w_j)$ as a sibling. This proves (I3). Lastly, (I4) holds for $G_{t+1}$ because all diagonal edges in $G_{t+1}$ are in $G_t$, and each unused $G_{t+1}$-sibling is in $G_{t}$.
\end{proof}

\subsection{The General Case}
\label{ss:general}
The key idea of \cref{thm:outerplanar} is that if the edge $(v_i, w_i) (v_j, w_j)$ is diagonal, then we embed $g_{t+1}$ in $(v_i, w_j)$ or $(v_j, w_i)$ to avoid creating a triangle in one of the factors. This is possible because of the additional inductive hypotheses (I3) and (I4).
In this section, we prove \cref{thm:dp_ub_stw} using a generalisation of the above technique to `diagonal cliques', defined below.

Let $\vec H$ be a digraph. The arcs of $\vec H$ are directed \defn{transitively} if $\overrightarrow{v_i v_j} \in E(\vec H)$ and $\overrightarrow{v_j v_k} \in E(\vec H)$ implies $\overrightarrow{v_i v_k} \in E(\vec H)$. A \defn{sink} in a digraph $\vec H$ is a vertex with outdegree $0$. A \defn{tournament} $\vec T$ is a digraph where exactly one of $\overrightarrow{uv}, \overrightarrow{vu}$ is in $E(\vec T)$ for any distinct vertices $u, v \in V(\vec T)$.
\begin{obs}
    Every transitively directed digraph has a sink.
\end{obs}

\begin{obs}
    \label{obs:induced_tourn}
    If $G \subsetsim \vec H_1 \, \squareslash \, \vec H_{2}$, and $C$ is a clique of $G$, then $P_1(C)$ and $P_2(C)$ induce cliques in $H_1$ and $H_2$ respectively. If further $\vec H_1$ and $\vec H_2$ are oriented, then $P_1(C)$ and $P_2(C)$ induce tournaments in $\vec H_1$ and $\vec H_2$ respectively.
\end{obs}

\begin{obs}
    \label{obs:induced_transitive}
    Any induced digraph of a transitive digraph is transitive.
\end{obs}

Let $\vec T$ be a transitive tournament. An arc $\overrightarrow{uv} \in E(\vec T)$ is a \defn{big arc} if $v$ is a sink for $\vec T$ and $u$ has outdegree $1$, implying that the only outneighbour of $u$ is $v$. Observe that the big arc of a transitive tournament is unique.

\begin{obs}
    \label{obs:big_edge}
    Let $G \subsetsim \vec H_1 \, \squareslash \, \vec H_2$, and suppose $D$ is a diagonal clique of $G$, with  $V(D) = \{(v_1, w_1), \dots, (v_k, w_k) \}$. Suppose the projections of $D$ induce transitive tournaments $\vec T_1$ and $\vec T_2$ in $\vec H_{1}$ and $\vec H_{2}$ respectively. Then there exists a unique edge $(v_i, w_i) (v_j, w_j) \in E(D)$ such that $\overrightarrow{v_i v_j}$ is the big arc for $\vec T_1$, and $\overrightarrow{w_i w_j}$ is the big arc for $\vec T_2$.
\end{obs}
We call the edge $(v_i, w_i) (v_j, w_j) \in E(D)$ in \cref{obs:big_edge} the \defn{big diagonal edge of $D$}, and $(v_i, w_j)$ and $(v_j, w_i)$ are the \defn{big siblings of $D$}. Note that the big arc is an arc in a directed graph, while the big diagonal edge is an edge in an undirected graph.

Let $G \subsetsim \vec H_1 \, \squareslash \, \vec H_2$. Suppose further that $\tw(G) = k$, and the projections of each diagonal $k$-clique in $G$ induce transitive tournaments in $\vec H_1$ and $\vec H_2$. Define the \defn{$G$-big siblings} to be:
\begin{equation*}
    \{(v_i, w_j), (v_j, w_i) : (v_i, w_i) (v_j, w_j) \text{ is the big diagonal edge of a diagonal $k$-clique in $G$}\}.
\end{equation*}

\restatedpubstw*

\begin{proof}
    By \cref{lem:edge_max_stw_k}, it suffices to prove the theorem for simple $k$-trees, which are constructed recursively from $K_{k+1}$ by repeatedly adding a new vertex so that its neighbourhood is an existing $k$-clique, as long as the same $k$-clique is not used more than once. Let $G_t$ be the graph at the $t$-th step of this process. So $G_1=K_{k+1}$.

    We recursively construct oriented digraphs $\vec H_{1,t}$ and $\vec H_{2,t}$ such that the following invariants are maintained:
    \begin{enumerate}[label=\textbf{(I\arabic*):},leftmargin=4em]
        \item $G_t \subsetsim \vec H_{1,t} \, \squareslash \, \vec H_{2,t}$.
        \item For each $i \in \{ 1, 2 \}$, $\indeg(\vec H_{i, t}) \leq k-1$.
        \item For each $i \in \{ 1, 2 \}$, $\tw(H_{i,t}) \leq k -1$.
        \item For each $k$-clique $C$ of $G_t$, the digraphs induced by $P_1(C)$ and $P_2(C)$ in $\vec H_{1,t}$ and $\vec H_{2,t}$ respectively are transitive tournaments.
        \item For each diagonal $k$-clique $D$ in $G_t$, each used big sibling of $D$ is adjacent to all vertices of $D$ in $G_t$. (The big siblings of $D$ are well-defined because of (I4)).
        \item No two diagonal $k$-cliques in $G_t$ share a common unused big sibling.
    \end{enumerate}
    When $t=1$, $G_1=K_{k+1}$. We define the factor graph $\vec H_{1,1}$ by $V(\vec H_{1,1}) = \{1, \dots, k\}$ and $E(\vec H_{1,1}) = \{\overrightarrow{ij}: 1 \leq i < j \leq k\}$. Define $\vec H_{2,1}$ by $V(\vec H_{2,1}) = \{1,2\}$ and $E(\vec H_{2,1}) = \{\overrightarrow{12}\}$. Then (I2) and (I3) hold. Since $\vec H_{1,1}$ and $\vec H_{2,1}$ are transitive tournaments, (I4) is satisfied. (I1) holds because we may embed $V(G_1)$ in the $k+1$ coordinates $\{(1,1), \dots, (k, 1), (k,2)\}$. Lastly, (I5) and (I6) hold because there are no diagonal cliques in $G_1$.

    Suppose we have constructed $G_t \subsetsim \vec H_{1, t} \, \squareslash \, \vec H_{2,t}$ satisfying (I1)--(I6). Let $g_{t+1}$ be the new vertex in $V(G_{t+1}) \setminus V(G_t)$ and suppose its neighbourhood is $C$, where $C$ is a $k$-clique in $G_t$. Let $(v_1, w_1), (v_2, w_2), \dots, (v_k, w_k)$ be the vertices of $C$. It is not necessary that the vertices $\{v_1, \dots, v_k\}$ are distinct, and likewise for $\{w_1, \dots, w_k\}$.

    \textbf{Case 1:} Suppose that $C$ is not diagonal. Then not all edges of $E(C)$ are diagonal, so there are indices $1 \leq i < j \leq k$ such that $v_i = v_j$ or $w_i = w_j$. Without loss of generality, suppose $v_i = v_j$. By (I4), the digraph induced by $\{w_1, \dots, w_k\}$ in $\vec H_{2, t}$ is a transitive tournament, and so has a sink $w_q$ for some $q \in \{1, \dots, k\}$.  Define $\vec H_{1, t+1}$ to be the directed graph obtained from $\vec H_{1, t}$ by adding a new vertex $x$ and arcs $\overrightarrow{v_1 x}, \dots, \overrightarrow{v_k x}$ directed towards $x$.
    Let $\vec H_{2, t+1}:= \vec H_{2, t}$.
    We embed $g_{t+1}$ at $(x, w_q)$ in $\vec H_{1, t+1} \, \squareslash \, \vec H_{2, t+1}$.
    We check that (I1)--(I6) continue to hold for $G_{t+1}$.

    \begin{enumerate}[label=\textbf{(I\arabic*):},leftmargin=3em]
        \item  It suffices to check that each edge in $E(G_{t+1}) \setminus E(G_t)$ is embedded in the directed product.
              That is, for each $i \in \{1, \dots, k\}$, $(v_i, w_i) (x, w_q)$ is an edge of $\vec H_{1, t+1} \, \squareslash \, \vec H_{2, t+1}$. If $w_i = w_q$, then $(v_i, w_i) (x, w_q)$ is an edge in the directed product since $v_i x \in E(H_{1, t+1} ) $. If $w_i \neq w_q$, then because $x$ and $w_q$ are sinks for $P_1(C+g_{t+1})$ and $P_2(C+g_{t+1})$ respectively, $\overrightarrow{v_i x} \in E(\vec H_{1, t+1})$ and $\overrightarrow{w_i w_q} \in E(\vec H_{2, t+1})$. Hence, $(v_i, w_i)(x, w_q)$ is an edge in the directed product.
        \item Since $v_i = v_j$, the set $\{v_1, \dots, v_k\}$ contains at most $k-1$ distinct vertices. Therefore, $x$ has indegree at most $k-1$. Further, the indegree of each vertex apart from $x$ in $\vec H_{1, t+1}$ remains unchanged.
        \item The graph $H_{1,t+1}$ is obtained from $H_{1,t}$ by adding a new vertex adjacent to each vertex in the clique $P_1(C)$, which has size at most $k-1$, therefore $\tw(H_{1, t+1}) \leq k-1$.
        \item The digraphs induced by $P_1(C+g_{t+1})$ and $P_2(C+g_{t+1})$ in $\vec H_{1, t+1}$ and $\vec H_{2, t+1}$ are transitive tournaments respectively.
              \cref{obs:induced_transitive} implies each new $k$-clique induces transitive tournaments on the factors.
    \end{enumerate}
    Suppose all diagonal $k$-cliques in $G_{t+1}$ are in $G_t$. To show (I5), we have to check that for each diagonal $k$-clique $D$ in $G_{t+1}$, no unused big sibling of $D$ in $G_t$ is used in $G_{t+1}$. The only new used vertex is $(x, w_q)$, and $(x, w_q)$ is not a big sibling of any existing diagonal $k$-clique because $x$ is a new vertex. Finally, (I6) continues to hold because each unused $G_{t+1}$-big sibling is in $G_{t}$.

    Therefore, we may assume that there is a diagonal $k$-clique $D$ in $G_{t+1}$ not contained in $G_t$. We claim:
    \begin{claim}
        \label{claim:diagonal}
        $(x, w_q) \in D$, $C$ is $(k-1, k)$-projected, and $w_i = w_q$ or $w_j = w_q$. If $w_i = w_q$, then $C \setminus D = \{ (v_i, w_i) \}$. If $w_j = w_q$, then $C \setminus D = \{ (v_j, w_j) \}$.
    \end{claim}
    \begin{proof}
        Observe that because $D$ is not contained in $G_t$, $(x, w_q) \in D$. Observe also that $|C \setminus D| = 1$. Suppose $C$ is $(p,q)$-projected. Then $C + g_{t+1}$ is $(p+1, q)$-projected. Since $D$ is a subclique of $C + g_{t+1}$, we have $q = k$ and $p + 1  = k$. This proves that $C$ is $(k-1, k)$-projected. Since $D$ is diagonal and $v_i = v_j$, $C \setminus D$ contains one of $(v_i, w_i)$ or $(v_j, w_j)$. As $(x, w_q) \in D$, and $D$ is diagonal, we have $(v_q, w_q) \in C \setminus D$. As $|C \setminus D| = 1$, it follows that $(v_i, w_i) = (v_q, w_q)$ or $(v_j, w_j) = (v_q, w_q)$, so $w_i= w_q$ or $w_j = w_q$.
    \end{proof}
    Without loss of generality, by \cref{claim:diagonal}, assume $w_j = w_q$. Then the vertices of $D$ are $ \{(v_1, w_1), \dots, (v_{j-1}, w_{j-1}), (v_{j+1}, w_{j+1}), \dots, (v_k, w_k), (x, w_q) \}$, and $D$ is the only diagonal $k$-clique in $G_{t+1}$ not in $G_t$.
    \begin{claim}\label{claim:big_diagonal_edge}
        $(v_i, w_i)(x, w_q)$ is the big diagonal edge of $D$.
    \end{claim}
    \begin{proof}
        By (I4) for $G_{t+1}$, the projections of $D$ induce transitive tournaments $\vec T_1$ and $\vec T_2$ in $\vec H_{1,t+1}$ and $\vec H_{2,t+1}$ respectively.
        By construction, $x$ is a sink for $\vec T_1$ and $w_q$ is a sink for $\vec T_2$. Hence, it suffices to show that $v_i$ and $w_i$ have outdegree $1$ in $\vec T_1$ and $\vec T_2$ respectively. Consider a vertex $(v_{z}, w_z) \in D \setminus \{(v_i, w_i), (x, w_q)\}$. Since $D$ is diagonal, $v_z \notin \{v_i, x\}$ and $w_z \notin \{w_i, w_q\}$.

        As $w_j = w_q$, and $w_q$ is a sink for $\vec T_2$, there is an arc from $w_z$ to $w_j$ in $E(\vec T_2)$. However, since $(v_z, w_z)(v_j, w_j) \in E(C)$, the definition of the directed product implies that there is an arc from $v_z$ to $v_j$.
        But $v_j = v_i$. Thus, for each $v_z \in P_1(D) \setminus \{ v_i, x \}$, there is an arc from $v_z$ to $v_i$. Therefore, $v_i$ has outdegree $1$ in $\vec T_1$, and the only outneighbour of $v_i$ is $x$. As $(v_z, w_z)(v_i, w_i) \in E(C)$ and $\overrightarrow{v_z v_i} \in E(\vec T_1)$, the definition of the directed product implies that $\overrightarrow{w_z w_i} \in E(\vec T_2)$. Therefore, $w_i$ has outdegree $1$, and the only outneighbour of $w_i$ is $w_q$, as desired.
    \end{proof}
    \begin{enumerate}[label=\textbf{(I\arabic*):},leftmargin=3em,resume]
        \item By \cref{claim:big_diagonal_edge}, $(v_i, w_q)$ is a used big sibling of $D$, and $(x, w_i)$ is an unused big sibling of $D$. We check that for each diagonal $k$-clique in $G_{t+1}$ that is not $D$, no unused $G_t$-big sibling is used in $G_{t+1}$. The only new used vertex is $(x, w_q)$, which is not a big sibling of an existing diagonal $k$-clique because $x$ is a new vertex. For $D$, $(v_i, w_q) = (v_j, w_j)$ is adjacent in $G_{t+1}$ to all vertices of $D$ by construction.
        \item By \cref{claim:big_diagonal_edge}, the only unused $G_{t+1}$-big sibling not in $G_t$ is $(x, w_i)$, which is not a big sibling of an existing diagonal $k$-clique in $G_t$ since $x$ is a new vertex.
    \end{enumerate}

    \textbf{Case 2:} Suppose that at step $t+1$, the $k$-clique $C$ is diagonal. Since (I4) and (I5) hold for $G_t$, there is a unique big diagonal edge $(v_i, w_i)(v_j, w_j) \in E(C)$ satisfying (I5) and (I6). We now claim that $(v_i, w_j)$ or $(v_j, w_i)$ is not used in $G_t$. Suppose both $(v_i, w_j)$ and $(v_j, w_i)$ are used in $G_t$. By (I5), $(v_i, w_j)$ and $(v_j, w_i)$ are both adjacent in $G_t$ to each vertex of $C$. Since $g_{t+1}$ is adjacent to each vertex in $C$, $K_{\overline{k},3}$ is a subgraph of $G_{t+1}$, contradicting the fact that $\stw(G_{t+1}) \leq k$.

    Without loss of generality, assume $(v_i, w_j)$ is not used in $G_t$.  For each $i \in \{ 1, 2 \}$, define $\vec H_{i, t+1} := \vec H_{i, t}$. Embed $g_{t+1}$ at $(v_i, w_j)$. Then (I2) and (I3) hold. (I4) is true because for each $i \in \{1, 2\}$, $P_i(C+g_{t+1}) = P_i(C)$, and property (I4) for $G_t$ implies that for each $i \in \{1, 2\}$, $P_i(C)$ induces a transitive tournament on $\vec H_{i, t}$. The addition of $(v_i, w_j)$ creates $k$ new cliques on $k$ vertices; however, none of them are diagonal. We verify (I1). We have to check that for each $z \in \{1, \dots, k\}$, $(v_i, w_j)$ is adjacent to $(v_z, w_z)$ in $\vec H_{1, t+1} \, \squareslash \, \vec H_{2, t+1}$.  Let $\vec T_1$ and $\vec T_2$ be the projections of $C$ on $\vec H_{1, t+1}$ and $\vec H_{2, t+1}$ respectively. Since $v_iv_j \in E(T_1)$ and $w_iw_j \in E(T_2)$, $(v_i, w_j)$ is adjacent to $(v_i, w_i)$ and $(v_j, w_j)$. Since $(v_i, w_i)(v_j, w_j)$ is the big diagonal edge for $C$, it follows that for each $z \in \{1, \dots, k\} \setminus \{i, j\}$, $\overrightarrow{v_zv_i} \in E(\vec T_1)$ and $\overrightarrow{w_z w_j} \in E(\vec T_2)$, so $(v_z, w_z)(v_i, w_j)$ is an edge in the directed product.

    To check (I5), note that all diagonal $k$-cliques in $G_{t+1}$ are already in $G_t$, so it suffices to show that for each diagonal $k$-clique other than $C$, an unused $G_t$-big sibling remains unused in $G_{t+1}$. The only used $G_{t+1}$-big sibling not in $G_t$ is $(v_i, w_j)$, and it is an unused big sibling of $C$ in $G_t$. By (I6) for $G_t$, no other diagonal $k$-clique in $G_t$ has $(v_i, w_j)$ as a big sibling. This proves (I5) for $G_{t+1}$. Lastly, (I6) holds for $G_{t+1}$ because all diagonal $k$-cliques in $G_{t+1}$ are in $G_t$, and each unused $G_{t+1}$-big sibling is in $G_t$.
\end{proof}

\section{Lower Bounds on Treewidth}
\label{s:optimal_bounds}
This section is dedicated to proving \cref{thm:dp_lb}. We in fact prove \cref{thm:main_thm_opt}, which is a strengthening of \cref{thm:dp_lb}. If $C$ is a clique of $G$ and $v$ is a vertex such that $C \subseteq N_G(v)$, then we say that the vertex $v$ is \defn{attached} to $C$. Implicit in the definition is that if $v$ is attached to $C$, then $v \notin C$. For a clique $C$, let \defn{$\att_G(C)$} denote the number of vertices attached to $C$ in $G$.

Consider a graph $G$ with a fixed embedding into $\vec H_1 \, \squareslash \, \vec H_2$. Retain all previous notation from the start of \cref{s:preliminaries}. See \cref{fig:diagonal} for a visual description of the following definitions. For a clique $C$, define the \defn{$P_1(C)$-strip} (resp.\ \defn{$P_2(C)$-strip}) as $P_1(C) \times V(H_2)$ (resp.\ $V(H_1) \times P_2(C)$).
\begin{itemize}
    \item A vertex $v \in V(G) \setminus C$ is \defn{$C$-diagonal} if it is attached to $C$ and it is not in the $P_1(C)$-strip or the $P_2(C)$-strip.
    \item A vertex $v \in C$ is \defn{$(C,P_1)$-redundant} (resp. \defn{$(C,P_2)$-redundant}) if there is a vertex $w \in C$ with $w \neq v$ such that $P_1(w) = P_1(v)$ (resp. $P_2(w) = P_2(v)$). A vertex $v \in C$ is \defn{$C$-redundant} if it is both $(C,P_1)$-redundant and $(C,P_2)$-redundant.
    \item A vertex $v \in C$ is \defn{$(C,P_1)$-attractive} (resp. \defn{$(C,P_2)$-attractive}) if every arc between $P_1(C) - P_1(v)$ (resp. $P_2(C) - P_2(v)$) and $P_1(v)$ (resp. $P_2(v)$) is directed towards $P_1(v)$ (resp. $P_2(v)$).
    \item A vertex $v \in V(G) \setminus C$ is \defn{$(C,P_1)$-magnetic} (resp. \defn{$(C,P_2)$-magnetic}) if $v \notin P_1(C) \times P_2(C)$, $v$ lies in the $P_2(C)$-strip (resp. $P_1(C)$-strip), $v$ is attached to $C$, and $v$ is $(C + v, P_1)$-attractive (resp. $(C+v, P_2)$-attractive).
\end{itemize}
Lastly, let \defn{$\mathcal{C}(G,k)$} denote the set of all $k$-cliques of a graph $G$.
\begin{obs}
    \label{obs:redundant}
    For a clique $C$, a vertex $v$ attached to $C$ and in the $P_1(C)$-strip is $(C+v, P_1)$-redundant.
\end{obs}

\begin{obs}
    \label{obs:diagonal}
    Let $G \subsetsim \vec H_1 \, \squareslash \, \vec H_2$, and let $C$ be a $(p,q)$-projected clique of $G$. The following are equivalent:
    \begin{itemize}
        \item $C$ is diagonal,
        \item $|p| = |q| = |C|$,
        \item $C - v$ is $(p-1, q-1)$-projected for each vertex $v \in C$.
    \end{itemize}
\end{obs}

\begin{lem}
    \label{lem:magnetic_attachments}
    Let $G \subsetsim \vec H_1 \, \squareslash \, \vec H_2$, and suppose $C$ is a clique of $G$ containing a $(C, P_1)$-attractive vertex $v$. If $w \in V(G) \setminus C$ is $(C, P_2)$-magnetic, then $P_1(v) = P_1(w)$.
\end{lem}
\begin{proof}
    Suppose not. Since $w$ is $(C, P_2)$-magnetic, every arc between $P_2(C)$ and $P_2(w)$ is directed towards $P_2(w)$. Since $w$ is attached to $C$, $vw \in E(G)$, and there is an arc directed from $P_1(v)$ towards $P_1(w)$. But $P_1(w) \in P_1(C) - P_1(v)$. This contradicts the assumption that $v$ is $(C, P_1)$-attractive.
\end{proof}

\begin{figure}[!b]
    \centering
    \includegraphics{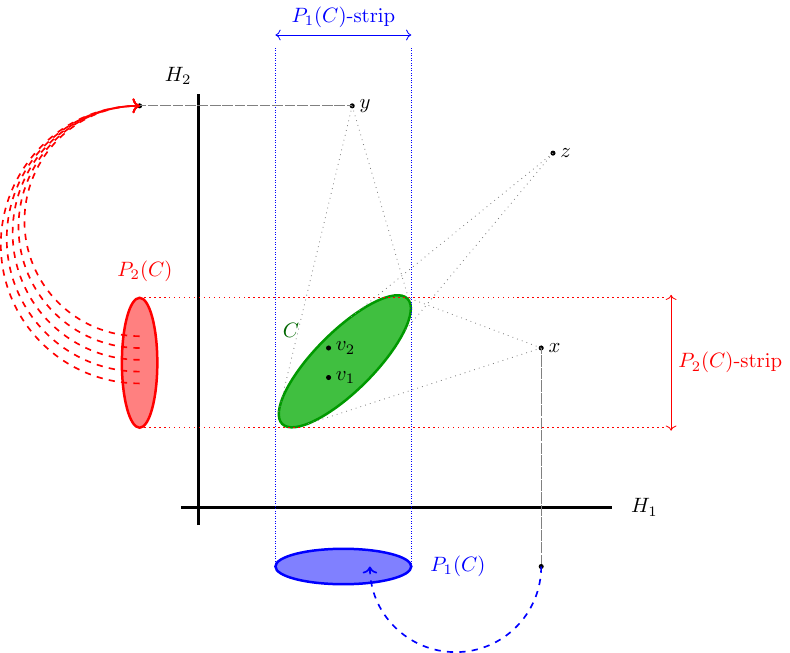}
    \caption{The blue and red dotted lines outline the $P_1(C)$ and $P_2(C)$-strip respectively.
        The vertices $v_1$ and $v_2$ are $(C, P_1)$-redundant. The blue and red ellipses represent $P_1(C)$ and $P_2(C)$ respectively. The vertex $x$ is not $(C, P_1)$-magnetic because there is at least one arc directed from its $H_1$-projection to $P_1(C)$. The vertex $y$ is $(C, P_2)$-magnetic because all arcs from $P_2(C)$ are directed towards its $H_2$-projection.
        The vertex $z$ is $C$-diagonal because it lies neither in the $P_1(C)$-strip nor the $P_2(C)$-strip.
        The vertices $x, y$ and $z$ are attached to $C$ (shown by the grey dotted lines).}
    \label{fig:diagonal}
\end{figure}

The following lemma allows us to make use of the bounds on $\indeg(\vec H_1)$ and $\indeg(\vec H_2)$.

\begin{lem}
    \label{lem:bad_vertices}
    Let $s, t, k \geq 1$. Let $G$ be a graph, and let $C$ be a $k$-clique of $G$. Let $f(s,t,k):= 2k^2 \max\{s, t\} + k^2 + 1$. Suppose $\att_G(C) \geq f(s,t,k)$. If $G \subsetsim \vec H_1 \, \squareslash \, \vec H_2$ with $\indeg(\vec H_1) \leq s$ and $ \indeg(\vec H_2) \leq t$, then there is a vertex $v$ in $G$ attached to $C$ that is either $C$-diagonal, $(C, P_1)$-magnetic, or $(C, P_2)$-magnetic.
\end{lem}
\begin{proof}
    Suppose $C$ is $(p,q)$-projected in $G \subsetsim \vec H_1 \, \squareslash \, \vec H_2$. Of the vertices attached to $C$, at least $\att_G(C) - pq \geq 2k^2 \max\{s, t\} + 1$ vertices are embedded outside $P_1(C) \times P_2(C)$. If no vertex is $C$-diagonal, then by the pigeonhole principle, there are at least $k^2 \max\{s, t\} + 1$ vertices in $(P_1(C) \times V(H_2)) \setminus (P_1(C) \times P_2(C))$ or $(V(H_1)\times P_2(C)) \setminus (P_1(C) \times P_2(C))$.
    Without loss of generality, suppose the former case happens. Since $|P_1(C)| \leq k$, these vertices project onto at least $k\max\{s, t\} + 1$ distinct vertices in $H_2$. Since $\indeg(H_2) \leq t$ and $\left|P_2(C)\right| \leq k$, it is not possible for all of these projections to have an arc directed towards $P_2(C)$. Therefore, there is a vertex that is $(C, P_2)$-magnetic.
\end{proof}

\begin{restatable}{thm}{restatemainthmopt}
    \label{thm:main_thm_opt}
    For any integers $s, t, k \geq 1$, there is a graph $G$ of treewidth $k$ such that
    for any digraphs $\vec H_1$ and $\vec H_2$ with $\indeg(\vec H_1) \leq s$ and $\indeg(\vec H_2) \leq t$, if $G \subsetsim \vec H_1 \, \squareslash \, \vec H_2$, then there is a $(k+1)$-clique of $G$ that projects onto $k+1$ distinct vertices in $\vec H_1$ or $\vec H_2$.
\end{restatable}
\cref{thm:main_thm_opt,obs:induced_tourn} imply that $\omega(H_1) \geq k+1$ or $\omega(H_2) \geq k+1$, where \defn{$\omega(G)$} is the \defn{clique number} of $G$. Therefore, $\tw(H_1) \geq \omega(H_1)- 1 \geq k$ or $\tw(H_2) \geq \omega(H_2) - 1 \geq k$, implying \cref{thm:dp_lb}. Since any graph with treewidth $k$ has simple treewidth at most $k+1$, \cref{thm:main_thm_opt} also implies the treewidth bounds in \cref{thm:dp_ub_stw} are optimal.

\begin{proof}
    Suppose to the contrary that for every graph $G$ with $\tw(G) = k$, there are digraphs $\vec H_1$ and $\vec H_2$ such that:
    \begin{enumerate}[label=\textbf{(C\arabic*):},leftmargin=4em]
        \item $G \subsetsim \vec H_1 \, \squareslash \, \vec H_2$,
        \item $\indeg(\vec H_1) \leq s$, and $ \indeg(\vec H_2) \leq t$, and
        \item each clique $C$ of $G$ satisfies $\left|P_1(C)\right| \leq k$ and $\left|P_2(C)\right| \leq k$.
    \end{enumerate}
    The condition (C3) implies that there is a minimum integer \defn{$\ell$} such that each clique $C$ of $G$ satisfies $|P_1(C)| + |P_2(C)| \leq \ell$. The integer $\ell$ is well-defined since $\ell \leq 2k$. Therefore, we may assume additionally that:
    \begin{enumerate}[label=\textbf{(C\arabic*):},leftmargin=4em,resume]
        \item each clique $C$ of $G$ satisfies $|P_1(C)| + |P_2(C)| \leq \ell$.
    \end{enumerate}
    A clique $C$ of $G$ is \defn{full} if $|P_1(C)| + |P_2(C)| = \ell$.

    The minimality of $\ell$ implies that there is a graph $G_0$ with treewidth $k$ such that for any digraphs $\vec H_1$ and $\vec H_2$ satisfying (C1)--(C4), some clique $C_0$ of $G_0$ is full. Without loss of generality, let $G_0$ be edge-maximal, so that every clique of $G_0$ is a subset of a $(k+1)$-clique. Hence, we may assume $|C_0| = k+1$.

    We now construct graphs $G_1, \dots, G_k$, where $G_i$ is constructed from $G_{i-1}$.
    For each $k$-clique $C \in \mathcal{C}(G_{i-1},k)$, let $X(C,i) = \{v_1^{C,i}, \dots, v_{f(s,t,k)}^{C,i}\}$, where $f(s,t,k)$ is the function in \cref{lem:bad_vertices}. Note that $f(s,t,k) \geq k^2 + 1$. The graph $G_i$ is defined by:
    \begin{align*}
        V(G_i) & = V(G_{i-1}) \quad \cup \bigcup_{C \in \mathcal{C}(G_{i-1},k)} X(C,i),                           \\
        E(G_i) & = E(G_{i-1}) \quad \cup \bigcup_{C \in \mathcal{C}(G_{i-1},k)} \{vw :  v  \in X(C,i), w \in C\}.
    \end{align*}
    This construction gives a chain of edge-maximal treewidth $k$ graphs: $G_0 \subseteq G_1 \subseteq \dots \subseteq G_{k}$. By assumption, there are digraphs $\vec H_1$ and $\vec H_2$ satisfying (C1)--(C4) for $G_{k}$.

    \begin{claim}
        \label{claim:k_cliques}
        No $k$-clique of $G_{k - 1}$ is full.
    \end{claim}
    \begin{proof}
        Suppose to the contrary that $C$ is a $k$-clique of $G_{k-1}$ that is $(p,q)$-projected with $p + q = \ell$. By construction, there are $f(s,t,k) \geq k^2 + 1 \geq pq + 1$ vertices attached to $C$ in ${V(G_{k}) \setminus V(G_{k-1})}$. So at least one of the vertices $v \in {V(G_{k}) \setminus V(G_{k-1})}$ attached to $C$ is embedded outside $P_1(C) \times P_2(C)$. The clique $C + v$ is then either $(p+1, q)$-projected, $(p, q+1)$-projected or $(p+1, q+1)$-projected, contradicting (C4).
    \end{proof}
    \begin{claim}
        \label{claim:redu_full}
        No full $(k+1)$-clique $C$ of $G_{k-1}$ contains a $C$-redundant vertex.
    \end{claim}
    \begin{proof}
        Suppose $C$ is a full $(k+1)$-clique of $G_{k-1}$, and $v \in C$ is $C$-redundant. Then $|P_1(C - v)| = |P_1(C)|$ and $|P_2(C - v)| = |P_2(C)|$. So $C-v$ is a full $k$-clique in $G_{k-1}$, contradicting \cref{claim:k_cliques}.
    \end{proof}

    \begin{claim}
        \label{claim:first_step}
        There is a $k$-clique $C_0'$ of $G_0$ and a vertex $v_1\in V(G_1) \setminus V(G_0)$ that is $(C_0', P_1)$-magnetic or $(C_0', P_2)$-magnetic. Further, the clique $C_0' + v_1$ is full.
    \end{claim}
    \begin{proof}
        By the choice of $G_0$, there is a $(k+1)$-clique $C_0$ of $G_0$ that is full. Suppose $C_0$ is $(p',q')$-projected, so $p' + q' = \ell$. Since $p' \leq k$ and $q'\leq k$, $C_0$ is not diagonal. By \cref{obs:diagonal}, there is a vertex $v_0 \in C_0$ such that $C_0':=C_0 - v_0$ is either $(p'-1,q')$-projected, $(p', q'-1)$-projected, or $(p',q')$-projected. By \cref{claim:k_cliques}, $C_0'$ is not $(p',q')$-projected.

        Therefore, $C_0'$ is $(p'-1,q')$-projected or $(p',q'-1)$-projected. By construction of $G_1$, there are $f(s, t, k)$ vertices in $V(G_1) \setminus V(G_0)$ attached to $C_0'$, so (C2) and \cref{lem:bad_vertices} imply there is a vertex $v_1 \in V(G_1) \setminus V(G_0)$ attached to $C_0'$ such that $v_1$ is $C_0'$-diagonal, $(C_0', P_1)$-magnetic, or $(C_0', P_2)$-magnetic. If $v_1$ is $C_0'$-diagonal, then $C_0' + v_1$ is $(p',q'+1)$-projected or $(p'+1,q')$-projected. This contradicts (C4). Therefore, $v_1$ is $(C_0', P_1)$-magnetic or $(C_0', P_2)$-magnetic. As $C_0'$ is $(p'-1,q')$-projected or $(p',q'-1)$-projected, $C_0' + v_1$ is either $(p'+1,q'-1)$-projected, $(p',q')$-projected, or $(p'-1,q'+1)$-projected. In each case, $C_0' + v_1$ is full.
    \end{proof}
    Let $C_0'$ and $v_1 \in V(G_1) \setminus V(G_0)$ be as in \cref{claim:first_step}. We may assume that $v_1$ is $(C_0', P_1)$-magnetic, since the argument for the other case is symmetric.  Define $C_1:= C_0' + v_1$ and suppose $C_1$ is $(p,q)$-projected, so $p + q = \ell$. If $p = 1$, then because $|P_1(C_1)|\cdot |P_2(C_1)|\geq |C_1| = k+1$, we have $q = k+1$, contradicting (C3). So $p \geq 2$. Define $x:= k+2-p$, so $2 \leq x \leq k$.

    \begin{claim}
        \label{claim:recursive_step}
        There are cliques $C_1, C_2, \dots, C_x$ and vertices $v_1, v_2, \dots, v_x$ such that for each $i \in \{ 1, \dots, x \}$:
        \begin{enumerate}[label=\textbf{(I\arabic*):},leftmargin=4em]
            \item $C_i$ is a $(k+1)$-clique of $G_i$ that is $(p+i-1, q-i+1)$-projected, and hence full;
            \item $v_1, v_2, \dots, v_i \in C_i$ and $v_i \in V(G_i) \setminus V(G_{i-1})$;
            \item  $v_1, \dots, v_i$ have distinct $H_1$-projections, and $P_1(\{v_1, \dots, v_i\})\cap P_1(C_i \setminus \{ v_1, \dots, v_i \}) = \varnothing$;
            \item $v_i$ is $(C_i, P_1)$-attractive;
            \item $v_i$ is $(C_i, P_2)$-redundant.
        \end{enumerate}
    \end{claim}

    \begin{proof}
        We proceed by induction on $i$. When $i = 1$, $v_1$ and $C_1$ satisfy (I1)--(I5) by \cref{claim:first_step}. For the inductive step, suppose $1 \leq i \leq x-1$ and $C_i$, $v_1, v_2, \dots, v_{i} $ have been found. By (I1), $|P_1(C_i)| = p+i - 1 \leq k$. (I2) and (I3) imply that $|P_1(C_i \setminus \{ v_1, \dots, v_i \})| = |P_1(C_i) | - i \leq k-i$. However, $C_i \setminus \{ v_1, \dots, v_i \}$ is a $(k+1-i)$-clique. Since $i \leq x-1 \leq k-1$, we have $k+1-i \geq 2$. By the pigeonhole principle, there are two vertices in $C_i \setminus \{ v_1, \dots, v_i \}$ that share the same $H_1$-projection. Let one of the two vertices be $v_i'$. Then $|P_1(C_i - v_i')| = |P_1(C_i)| = p+i-1$. By (I1) and \cref{claim:k_cliques}, $C_i-v_i'$ is $(p+i-1, q-i)$-projected, implying that $P_2(v_i') \notin \{ P_2(v_1), \dots, P_2(v_i) \}$.

        By construction, there are $f(s, t, k)$ vertices in $V(G_{i+1}) \setminus V(G_i)$ attached to the clique $C_i - v_i'$ in $G_{i+1}$. \cref{lem:bad_vertices} implies that there is a vertex $v_{i+1} \in V(G_{i+1}) \setminus V(G_i)$ attached to $C_i - v_i'$ that is either $(C_i - v_i')$-diagonal, $(C_i - v_i', P_1)$-magnetic, or $(C_i - v_i', P_2)$-magnetic. The vertex $v_{i+1}$ is not $(C_i - v_i')$-diagonal, otherwise the clique $C_i - v_i' + v_{i+1}$ contradicts (C4).

        We now claim that $v_{i+1}$ is not $(C_i - v_i', P_2)$-magnetic.
        Suppose $v_{i+1}$ is $(C_i - v_i', P_2)$-magnetic. By (I4), $v_{i}$ is $(C_i - v_i', P_1)$-attractive. Therefore, \cref{lem:magnetic_attachments} implies that $P_1(v_{i+1}) = P_1(v_i)$.
        So $v_i$ is $(C_i - v_i' + v_{i+1}, P_1)$-redundant.
        By (I5), $v_i$ is $(C_i, P_2)$-redundant. So there is a vertex $w \in C_i -v_i$ such that $P_2(w) = P_2(v_i)$. Because $P_2(v_i') \notin \{ P_2(v_1), \dots, P_2(v_i) \}$, we have $P_2(v_i') \neq P_2(v_i)$. Therefore, $w \neq v_i'$ and $w \in C_i - v_i'$, and $v_i$ is $(C_i - v_i', P_2)$-redundant. Hence, $v_i$ is $(C_i - v_i' + v_{i+1}, P_2)$-redundant and therefore $v_i$ is $(C_i - v_i' + v_{i+1})$-redundant. As $v_{i+1}$ is assumed to be $(C_i - v_i', P_2)$-magnetic, $v_{i+1}$ is embedded in $(P_1(C_i - v_i') \times V(H_2)) \setminus (P_1(C_i - v_i') \times P_2(C_i - v_i'))$. Therefore, $C_i - v_i' + v_{i+1}$ is $(p+i-1, q-i+1)$-projected and hence full, but $C_i - v_i' + v_{i+1}$ contains a $(C_i - v_i' + v_{i+1})$-redundant vertex, namely $v_{i}$. This contradicts \cref{claim:redu_full}.

        Therefore, $v_{i+1}$ is $(C_i - v_i', P_1)$-magnetic. Set $C_{i+1}:= C_i - v_i' + v_{i+1}$. Then $C_{i+1}$ is $(p+i, q-i)$-projected. This shows (I1). By construction, (I2), (I3), (I4) are satisfied. (I5) is satisfied because $v_{i+1}$ is embedded in the $P_2(C_i - v_i')$-strip.
    \end{proof}
    By \cref{claim:recursive_step}, the clique $C_x$ is $(k+1, q - (k+2-p) + 1)$-projected, which contradicts (C3). This completes the proof.
\end{proof}

\section{Simple Treewidth and Strong Products}
\label{s:stw_strong_products}

This section is dedicated to proving \cref{thm:sp_lb_stw}. This is done by proving \cref{prop:stw_projections}, which in turn implies \cref{thm:sp_lb_stw}.
The proof of \cref{prop:stw_projections} uses the following lemma by \citet*[Lemma 4.6]{LNW}.
\begin{lem}
    \label{lem:lnw2}
    For every integer $k \geq 1$, there exists a  graph $G$ of treewidth $k$ such that for any graphs $H_1$ and $H_2$ with $\omega(H_1) \leq k$ and $\omega(H_2) \leq k$, if $G \subsetsim H_1 \boxtimes H_2$, then there is a $(p,q)$-projected $(k+1)$-clique in $G$ with $p + q \geq k + 3$.
\end{lem}
\begin{prop}
    \label{prop:stw_projections}
    For each integer $k \geq 3$, there exists a graph $G$ of simple treewidth $k$ such that for any graphs $H_1$ and $H_2$ with $\omega(H_1) \leq k-1$ and $\omega(H_2) \leq k-1$, if $G \subsetsim H_1 \boxtimes H_2$, then there is a $(p,q)$-projected clique in $G$ with $p + q \geq k + 3$.
\end{prop}

\begin{proof}
    \cref{lem:lnw2} implies that there is a graph $G_0$ with treewidth $k-1$ such that for any graphs $H_1$ and $H_2$ with $\omega(H_1), \omega(H_2) \leq k -1$, if $G_0 \subsetsim H_1 \boxtimes H_2$, then there is a $(p,q)$-projected $k$-clique such that $p + q \geq k+2$. We construct $G$ from $G_0$ in two steps:
    \begin{itemize}
        \item For each $k$-clique $C \in \mathcal{C}(G_0,k)$, let $X_{C} := \{v_1^C, v_2^C\}$ be a set of two new vertices. Define the graph $G_1$ by:
              \begin{align*}
                  V(G_1) & = V(G_0) \cup \bigcup_{C \in \mathcal{C}(G_0,k)} X_{C},                          \\
                  E(G_1) & = E(G_0) \cup \bigcup_{C \in \mathcal{C}(G_0,k)} \{vw :  v \in X_{C}, w \in C\}.
              \end{align*}
        \item For each $(k-1)$-clique $C' \in \mathcal{C}(G_1, k-1)$, let $X_{C'} = \{v_1^{C'}, \dots, v_{(k-1)^2 + 1}^{C'}\}$ be a set of $(k-1)^2 + 1$ new vertices. Define the graph $G$ by:
              \begin{align*}
                  V(G) & = V(G_1) \cup \bigcup_{C' \in \mathcal{C}(G_1, k-1)} X_{C'},                            \\
                  E(G) & = E(G_1) \cup \bigcup_{C' \in \mathcal{C}(G_1, k-1)} \{v w :  v \in X_{C'}, w \in C'\}.
              \end{align*}
    \end{itemize}
    \begin{claim}
        $\stw(G) \leq k$.
    \end{claim}
    \begin{proof}
        As $\tw(G_0) \leq k -1$, there is a tree-decomposition $(T_0, \mathcal{B}_0)$ of $G_0$ such that each bag has size at most $k$. We may assume that no two bags are identical. Since each bag has size at most $k$, if a $k$-clique $C$ of $G_0$ is contained in a bag $B_x \in \mathcal{B}_0$, then $B_x = C$. Since no two bags are identical, each $k$-clique $C$ of $G_0$ is contained in a single bag.

        We modify $(T_0, \mathcal{B}_0)$ to construct a tree-decomposition $(T_1, \mathcal{B}_1)$ of $G_1$. The vertices of each bag in $B_x \in \mathcal{B}_0$ either induce a $k$-clique $C$ in $G_0$, or they do not. In the former case, we modify $T_0$ by adding a leaf vertex $x'$ adjacent to $x$. We then add $v_1^C$ to $B_x$, and create a new bag $B_{x'}:= C \cup \{v_2^C\}$. In the latter case, leave $T_0$ and $B_x$ unchanged. Let $T_1$ be the resulting tree, and let $\mathcal{B}_1$ be the collection of all bags. Then $(T_1, \mathcal{B}_1)$ is a tree-decomposition of $G_1$ in which each bag of $\mathcal{B}_1$ has size at most $k+1$. Therefore, $\tw(G_1) \leq k$.

        To show that $\stw(G_1) \leq k$, we verify that for each set $S$ of $k$ vertices in $G_1$, there are at most two bags containing $S$. If there is a vertex $v \in S \setminus V(G_0)$, then $v = v_i^C$ for some $k$-clique $C$ of $G_0$ and some index $i \in \{1, 2\}$. By construction, $v_i^C$ appears in exactly one bag of $\mathcal{B}_1$.

        Otherwise, $S \subseteq V(G_0)$. Since each bag in $\mathcal{B}_0$ has size at most $k$, and no two bags are identical, there is at most one bag in $\mathcal{B}_0$ containing $S$. Consequently, there is at most one bag in $\{ B_{x} \in \mathcal{B}_1 : x \in V(T_0)  \}$ containing $S$. We claim that there is at most one bag in $\{ B_{x} \in \mathcal{B}_1 : x \in V(T_1) \setminus V(T_0) \}$ containing $S$. Suppose $S \subseteq B_{x'}$ and $S \subseteq B_{x''}$, where $x', x'' \in V(T_1) \setminus V(T_0)$ and $x' \neq x''$. By construction, $B_{x'} = C' \cup \{v_2^{C'}\}$ and $B_{x''} = C'' \cup \{v_2^{C''}\}$ for some distinct $k$-cliques $C'$ and $C''$ of $G_0$. Let $w \in C' \setminus C''$. Then $w, v_2^{C'} \in B_{x'} \setminus B_{x''}$, and because $|B_{x'}| = k+1$, $|B_{x'} \cap B_{x''}| \leq k-1$. This contradicts the assumption that $S \subseteq B_{x'} \cap B_{x''}$ and $\left|S \right|=k$.

        Finally, we claim that $\stw(G) \leq k$. We modify $(T_1, \mathcal{B}_1)$ to construct a tree-decomposition $(T, \mathcal{B})$ of $G$. Each $(k-1)$-clique $C'$ of $G_1$ is contained in some bag $B_x \in \mathcal{B}_1$ (with $x \in V(T_1)$). We modify $T_1$ by adding leaf vertices $x_1, \dots, x_{(k-1)^2 + 1}$ adjacent to $x$. For each $i \in \{1, \dots, (k-1)^2 + 1\}$, create a new bag $B_{x_i}:= C' \cup \{v_{i}^{C'}\}$, and note that $|B_{x_i}| \leq k$. Let $T$ be the resulting tree, and let $\mathcal{B}$ be the collection of all bags. Then $(T, \mathcal{B})$ is a tree-decomposition of $G$ in which each bag of $\mathcal{B}$ has size at most $k+1$. Therefore, $\tw(G) \leq k$. Since each new $k$-clique is in exactly one bag, and because $\stw(G_1) \leq k$, we have $\stw(G) \leq k$.
    \end{proof}
    Suppose for contradiction that there are graphs $H_1$ and $H_2$ such that $\omega(H_1), \omega(H_2) \leq k -1$, $G \subsetsim H_1 \boxtimes H_2$, and each clique of $G$ is $(p,q)$-projected with $p + q \leq k +2$. We say that a clique $C$ of $G$ is \defn{full} if it is $(p,q)$-projected with $p + q = k + 2$. In this language, \cref{lem:lnw} implies that there is a full $k$-clique $\mathdefn{C}$ of $G_0$.

    \begin{claim}
        \label{claim:k-1full}
        No $(k-1)$-clique of $G_1$ is full.
    \end{claim}
    \begin{proof}
        Suppose to the contrary that $C'$ is a $(p,q)$-projected $(k-1)$-clique of $G_1$ with $p+q = k+2$. Since $p, q \leq k -1$, there are $(k-1)^2 + 1 \geq pq + 1$ vertices attached to $C'$ in $V(G) \setminus V(G_1)$. At least one of the vertices $v \in V(G) \setminus V(G_1)$ is embedded outside $P_1(C') \times P_2(C')$. Then $C' + v$ is either $(p+1, q)$-projected, $(p, q+1)$-projected or $(p+1, q+1)$-projected, a contradiction.
    \end{proof}

    \begin{claim}
        \label{claim:k_red_vertex}
        No full $k$-clique $C$ of $G_1$ contains a $C$-redundant vertex.
    \end{claim}
    \begin{proof}
        Suppose $C$ is a full $k$-clique of $G_1$, and $v \in C$ is $C$-redundant. Then $|P_1(C - v)| = |P_1(C)|$ and $|P_2(C - v)| = |P_2(C)|$. So $C - v$ is a full $(k-1)$-clique in $G_1$, contradicting \cref{claim:k-1full}.
    \end{proof}
    By an \defn{ordering} of $C$ we mean a labelling of the vertices of $C$ with $\{v_1, \dots, v_k\}$. Fix an ordering of $C$.
    For each $i \in \{1, \dots, k\}$, let $Q_i:=\{v_1, \dots, v_i\}$. Let $p_i, q_i \geq 1$ be integers such that $Q_i$ is $(p_i, q_i)$-projected. Note that $p_i$ and $q_i$ depend on the ordering of $C$, but we always have that $p_1 =  q_1 = 1$, $Q_k = C$, $p_k = p$, and $q_k = q$.

    Suppose $(p_i, q_i) = (p_{i+1}, q_{i+1})$ for some index $i \in \{1, \dots, k-1\}$. Then $v_{i+1}$ is both $(C,P_1)$-redundant and $(C,P_2)$-redundant. This contradicts \cref{claim:k_red_vertex}.
    Hence, for each $i \in \{1, \dots, k-1\}$, either $p_{i+1} + q_{i+1} = p_i + q_i + 1$ or $p_{i+1} + q_{i+1} = p_i + q_i + 2$. In the latter case, we refer to $i+1 \in \{ 2, \dots k \}$ as a \defn{jump index}. Since $p + q = k + 2$, there is exactly one jump index $\mathdefn{j} \in \{2, \dots, k\}$. Therefore:
    \begin{claim}
        \label{claim:one_jump_index}
        Any ordering of $C$ has exactly one jump index.\qedhere
    \end{claim}
    Observe that $P_i(v_1) \neq P_i(v_j)$ for each $i \in \{ 1, 2 \}$. Let $Q':= \{v_2, \dots, v_k\} \setminus \{v_j\}$. Since $k \geq 3$, $Q' \neq \varnothing$. Define:
    \begin{align*}
        E_1:= \{ x \in Q' : P_1(v_1) = P_1(x) \}, \\
        E_2:= \{ x \in Q' : P_2(v_1) = P_2(x) \}, \\
        E_3:= \{ x \in Q' : P_1(v_j) = P_1(x) \}, \\
        E_4:= \{ x \in Q' : P_2(v_j) = P_2(x) \}.
    \end{align*}
    If $E_1 \neq \varnothing$ and $E_2 \neq \varnothing$, then $v_1$ is $C$-redundant, contradicting \cref{claim:k_red_vertex}. Therefore, $E_1$ or $E_2$ is empty.
    Similarly, $E_3$ or $E_4$ is empty. If $x \in E_1 \cap E_4$ or $x \in E_2 \cap E_3$, then $x$ is $C$-redundant, contradicting \cref{claim:k_red_vertex}. Hence, $E_1 \cap E_4  = E_2 \cap E_3 = \varnothing$. Since $P_i(v_1) \neq P_i(v_j)$ for each $i \in \{ 1, 2 \}$, $E_1 \cap E_3 = E_2 \cap E_4 = \varnothing$.

    \begin{claim}
        \label{claim:is_a_partition}
        Each vertex in $Q'$ shares a coordinate with $v_1$ or $v_j$, implying that $E_1 \cup E_2 \cup E_3 \cup E_4 = Q'$.
    \end{claim}
    \begin{proof}
        Suppose $v_z \in Q'$ does not share a coordinate with both $v_1$ and $v_j$. Then we may reorder $C$ so that the vertices $\{ v_1, v_j, v_z \}$ are the first three vertices in the new ordering, which has two jump indices. This contradicts \cref{claim:one_jump_index}.
    \end{proof}
    \begin{claim}
        \label{claim:disjoint_subclaim}
        Either:
        \begin{enumerate}[label=\textbf{(C\arabic*):},leftmargin=4em]
            \item $E_1 \cup E_4 = Q'$, $E_1 \cap E_4 = \varnothing$, and $E_2, E_3 = \varnothing$, or
            \item $E_2 \cup E_3 = Q'$, $E_2 \cap E_3 = \varnothing$, and $E_1, E_4 = \varnothing$.
        \end{enumerate}
    \end{claim}
    \begin{proof}
        By \cref{claim:is_a_partition}, $E_1 \cup E_4 \neq \varnothing$ or $E_2 \cup E_3 \neq \varnothing$. Suppose $E_1 \neq \varnothing$. Since $E_1$ or $E_2$ is empty, $E_2 = \varnothing$. Assume for contradiction that $E_3 \neq \varnothing$. Since $E_3$ or $E_4$ is empty, $E_4 = \varnothing$. Thus, $E_1 \cup E_3 = Q'$, and $E_1 \cap E_3 = \varnothing$. By \cref{claim:k_red_vertex}, since no vertex of $C$ is $C$-redundant, no vertex in $E_1$ shares the same $H_2$-projection as a vertex in $E_3$. This forces $|P_2(C)| = k$, a contradiction to the assumption that $\omega(H_2) \leq k - 1$. Therefore, $E_3 = \varnothing$. Then $E_1 \cup E_4 = Q'$, $E_1 \cap E_4 = \varnothing$, and we deduce (C1). If $E_4 \neq \varnothing$, the same argument concludes (C1). If $E_2 \neq \varnothing$ or $E_3 \neq \varnothing$, a symmetric argument concludes (C2).
    \end{proof}
    Without loss of generality, suppose case (C1) of \cref{claim:disjoint_subclaim} occurs. The assumption that $\omega(H_i) \leq k -1$ and $P_i(v_1) \neq P_i(v_j)$ for each $i \in \{ 1, 2 \}$ implies that $\left|E_1\right|, \left|E_4\right| \leq k -3$. Since $E_1 \cup E_4 = Q'$ and $E_1 \cap E_4 = \varnothing$, we have $|E_1| ,|E_4| \geq 1$.

    Let $x:=P_1(v_1) \in V(H_1)$ and $y:=P_2(v_j) \in V(H_2)$. \cref{claim:disjoint_subclaim} implies that for some $w_1, \dots, w_{p-1} \in V(H_1) -x$ and $u_1, \dots, u_{q-1} \in V(H_2) - y$, the vertices of $C$ have coordinates $(x, u_1), \dots, (x, u_{q-1}), (w_1, y), \dots, (w_{p-1}, y)$. Let $S_x:= \{(x, u_1), \dots, (x, u_{q-1})\}$ and $S_y := \{(w_1, y), \dots, (w_{p-1}, y)\}$. Observe that $S_x = E_1 \cup \{v_1\}$, $S_y = E_4 \cup \{v_j\}$, $|S_x| = q-1$, $|S_y| = p-1$, $C = S_x \cup S_y$, $S_x \cap S_y = \varnothing$, and $|S_x|, |S_y| \geq 2$.

    By construction of $G_1$, there are two vertices attached to $C$. One of these vertices, $z$, is not embedded at $(x,y)$. However, as $C$ is full, $z$ is embedded in $P_1(C) \times P_2(C)$. Suppose $z$ is embedded at $(w_i, u_j)$ for some $i \in \{1, \dots, p-1\}$ and $j \in \{1, \dots, q-1\}$. Let $C':= (C \cup (w_i, u_j) )\setminus \{ (x,u_j), (w_i, y)\}$, and note that $C'$ is a $(k-1)$-clique of $G_1$.
    \begin{claim}
        \label{claim:C'_full}
        $C'$ is full.
    \end{claim}
    \begin{proof}
        Since $|S_x| \geq 2$, $S_x \setminus \{(x, u_j)\}$ is non-empty. So $x \in P_1(C')$ and $|P_1(C')| = |P_1(C)| = p$. Similarly, since $|S_y| \geq 2$, we have $|P_2(C')| = |P_2(C)| = q$. Therefore, $C'$ is full.
    \end{proof}
    \cref{claim:C'_full} contradicts \cref{claim:k-1full}. We have thus arrived at a contradiction.
\end{proof}

\restatablesplbstw*

\begin{proof}
    Since $p, q \geq 1$ and $p,q \leq k -2$, $k \geq 3$. Thus, the hypothesis of \cref{prop:stw_projections} is satisfied. Let $G$ be the simple treewidth $k$ graph  given by \cref{prop:stw_projections}. Assume for contradiction that there are graphs $H_1$ and $H_2$ with $\tw(H_1) \leq p$ and $\tw(H_2) \leq q$ such that $G \subsetsim H_1 \boxtimes H_2$. Since $p, q \leq k- 2$, $\omega(H_1) \leq \tw(H_1) + 1 \leq p + 1 \leq k -1$ and $\omega(H_2) \leq \tw(H_2) + 1 \leq q +1 \leq k -1$. \cref{prop:stw_projections} implies that there is a $(p', q')$-projected clique of $G$ with $p' + q' \geq k + 3$. By \cref{obs:induced_tourn}, $\omega(H_1) \geq p'$ and $\omega(H_2) \geq q'$, so $\omega(H_1) + \omega(H_2) \geq p' + q' \geq k + 3$. On the other hand, the assumption of \cref{thm:sp_lb_stw} implies that $\omega(H_1) + \omega(H_2) \leq p + 1 + q + 1 \leq k + 2$. This is a contradiction.
\end{proof}

\section{A Conjecture about Planar Graphs}

We conclude with a conjecture. Recall that \cref{thm:Planar222} says that every planar graph is contained in $H_1 \StrongProd H_2 \StrongProd K_2$ for some graphs $H_1$ and $H_2$ with $\tw(H_1)\leq 2$ and $\tw(H_2)\leq 2$. The product in \cref{thm:Planar222} is dense, while planar graphs are sparse. So it would be interesting to prove the following qualitative strengthening, where the product is sparse.

\begin{conj}
    \label{conj:planar}
    There exists $c,d \geq 1$ such that every planar graph $G$ is contained in $(\vec H_1 \, \squareslash \, \vec H_2) \boxtimes K_c$, for some digraphs $\vec H_1$ and $\vec H_2$ with  $\tw(H_i)\leq 2$ and $\indeg(\vec H_i)\leq d$ for each $i \in \{ 1, 2 \}$.
\end{conj}
There is some evidence to support \cref{conj:planar}:
\begin{itemize}
    \item Outerplanar graphs are contained in directed projects of trees with indegree $1$ (\cref{thm:outerplanar}).
    \item \cref{conj:planar} holds for planar graphs with treewidth $3$ (\cref{thm:apollonian}), and many counterexamples in graph product structure theory are planar with treewidth $3$ \citep{distel2024treewidth, DJMMUW20}.
\end{itemize}

\subsubsection*{Acknowledgements}
Thanks to Nikolai Karol for helpful conversations about graph products.

    {
        \fontsize{10pt}{11pt}
        \selectfont
        \bibliographystyle{DavidNatbibStyle}
        \bibliography{ref}

\def\soft#1{\leavevmode\setbox0=\hbox{h}\dimen7=\ht0\advance \dimen7 by-1ex\relax\if t#1\relax\rlap{\raise.6\dimen7 \hbox{\kern.3ex\char'47}}#1\relax\else\if T#1\relax \rlap{\raise.5\dimen7\hbox{\kern1.3ex\char'47}}#1\relax \else\if d#1\relax\rlap{\raise.5\dimen7\hbox{\kern.9ex \char'47}}#1\relax\else\if D#1\relax\rlap{\raise.5\dimen7 \hbox{\kern1.4ex\char'47}}#1\relax\else\if l#1\relax \rlap{\raise.5\dimen7\hbox{\kern.4ex\char'47}}#1\relax \else\if L#1\relax\rlap{\raise.5\dimen7\hbox{\kern.7ex \char'47}}#1\relax\else\message{accent \string\soft \space #1 not defined!}#1\relax\fi\fi\fi\fi\fi\fi}
\begin{thebibliography}{28}
\providecommand{\natexlab}[1]{#1}
\providecommand{\msn}[1]{MR:\,\href{http://www.ams.org/mathscinet-getitem?mr=MR{#1}}{#1}}
\providecommand{\ZBL}[1]{Zbl:\,\href{https://www.zentralblatt-math.org/zmath/en/search/?q=an:#1}{#1}}
\providecommand{\url}[1]{\texttt{#1}}
\providecommand{\urlprefix}{}
\expandafter\ifx\csname urlstyle\endcsname\relax
  \providecommand{\doi}[1]{doi:\discretionary{}{}{}#1}\else
  \providecommand{\doi}{doi:\discretionary{}{}{}\begingroup \urlstyle{rm}\Url}\fi

\bibitem[{Arnborg and Proskurowski(1986)}]{AP-SJADM96}
\textsc{Stefan Arnborg and Andrzej Proskurowski}.
\newblock \href{https://doi.org/10.1137/0607033}{Characterization and recognition of partial $3$-trees}.
\newblock \emph{SIAM J. Algebraic Discrete Methods}, 7(2):305--314, 1986.

\bibitem[{Bang-Jensen and Gutin(2018)}]{bang2018classes}
\textsc{J{\o}rgen Bang-Jensen and Gregory Gutin}.
\newblock \href{https://link.springer.com/book/10.1007/978-3-319-71840-8}{Classes of directed graphs}, vol.~11.
\newblock Springer, 2018.

\bibitem[{Bodlaender(1998)}]{Bodlaender98}
\textsc{Hans~L. Bodlaender}.
\newblock \href{https://doi.org/10.1016/S0304-3975(97)00228-4}{A partial $k$-arboretum of graphs with bounded treewidth}.
\newblock \emph{Theoret. Comput. Sci.}, 209(1-2):1--45, 1998.

\bibitem[{Bonamy et~al.(2022)Bonamy, Gavoille, and Pilipczuk}]{BGP22}
\textsc{Marthe Bonamy, Cyril Gavoille, and Micha{\l} Pilipczuk}.
\newblock \href{https://doi.org/10.1137/20M1330464}{Shorter labeling schemes for planar graphs}.
\newblock \emph{SIAM J. Discrete Math.}, 36(3):2082--2099, 2022.

\bibitem[{Bonnet et~al.(2022)Bonnet, Kwon, and Wood}]{BKW}
\textsc{\'Edouard Bonnet, {O-joung} Kwon, and David~R. Wood}.
\newblock \href{https://arxiv.org/abs/2202.11858}{Reduced bandwidth: a qualitative strengthening of twin-width in minor-closed classes (and beyond)}.
\newblock 2022, arXiv:2202.11858.

\bibitem[{Chen(2016)}]{Chen16}
\textsc{Hao Chen}.
\newblock \href{https://doi.org/10.1007/s00454-016-9777-3}{Apollonian ball packings and stacked polytopes}.
\newblock \emph{Discrete Comput. Geom.}, 55(4):801--826, 2016.

\bibitem[{Diestel(2005)}]{Diestel05}
\textsc{Reinhard Diestel}.
\newblock \href{http://diestel-graph-theory.com/}{Graph theory}, vol. 173 of \emph{Graduate Texts in Mathematics}.
\newblock Springer, 3rd edn., 2005.

\bibitem[{Ding et~al.(1998)Ding, Oporowski, Sanders, and Vertigan}]{DOSV98}
\textsc{Guoli Ding, Bogdan Oporowski, Daniel~P. Sanders, and Dirk Vertigan}.
\newblock \href{https://doi.org/10.1007/s004930050001}{Partitioning graphs of bounded tree-width}.
\newblock \emph{Combinatorica}, 18(1):1--12, 1998.

\bibitem[{Distel et~al.(2025)Distel, Hendrey, Karol, Wood, and Yip}]{distel2024treewidth}
\textsc{Marc Distel, Kevin Hendrey, Nikolai Karol, David~R. Wood, and Jung~Hon Yip}.
\newblock \href{http://dx.doi.org/10.46298/dmtcs.14785}{Treewidth 2 in the planar graph product structure theorem}.
\newblock \emph{Discrete Mathematics \& Theoretical Computer Science}, 27:2:8, 2025.

\bibitem[{D\k{e}bski et~al.(2021)D\k{e}bski, Felsner, Micek, and Schr\"{o}der}]{DFMS21}
\textsc{Micha{\l} D\k{e}bski, Stefan Felsner, Piotr Micek, and Felix Schr\"{o}der}.
\newblock \href{https://doi.org/10.19086/aic.27351}{Improved bounds for centered colorings}.
\newblock \emph{Adv. Comb.}, \#8, 2021.

\bibitem[{Dujmovi\'c et~al.(2021)Dujmovi\'c, Esperet, Gavoille, Joret, Micek, and Morin}]{DEGJMM21}
\textsc{Vida Dujmovi\'c, Louis Esperet, Cyril Gavoille, Gwena\"el Joret, Piotr Micek, and Pat Morin}.
\newblock \href{https://doi.org/10.1145/3477542}{Adjacency labelling for planar graphs (and beyond)}.
\newblock \emph{J. ACM}, 68(6):42, 2021.

\bibitem[{Dujmovi{\'c} et~al.(2020{\natexlab{a}})Dujmovi{\'c}, Esperet, Joret, Walczak, and Wood}]{DEJWW20}
\textsc{Vida Dujmovi{\'c}, Louis Esperet, Gwena\"{e}l Joret, Bartosz Walczak, and David~R. Wood}.
\newblock \href{https://doi.org/10.19086/aic.12100}{Planar graphs have bounded nonrepetitive chromatic number}.
\newblock \emph{Adv. Comb.}, \#5, 2020{\natexlab{a}}.

\bibitem[{Dujmovi{\'c} et~al.(2020{\natexlab{b}})Dujmovi{\'c}, Joret, Micek, Morin, Ueckerdt, and Wood}]{DJMMUW20}
\textsc{Vida Dujmovi{\'c}, Gwena\"{e}l Joret, Piotr Micek, Pat Morin, Torsten Ueckerdt, and David~R. Wood}.
\newblock \href{https://doi.org/10.1145/3385731}{Planar graphs have bounded queue-number}.
\newblock \emph{J. ACM}, 67(4):\#22, 2020{\natexlab{b}}.

\bibitem[{Dvor{\'{a}}k et~al.(2022)Dvor{\'{a}}k, Gon{\c{c}}alves, Lahiri, Tan, and Ueckerdt}]{DGLTU22}
\textsc{Zdenek Dvor{\'{a}}k, Daniel Gon{\c{c}}alves, Abhiruk Lahiri, Jane Tan, and Torsten Ueckerdt}.
\newblock \href{https://doi.org/10.4230/LIPIcs.SoCG.2022.38}{On comparable box dimension}.
\newblock In \textsc{Xavier Goaoc and Michael Kerber}, eds., \emph{Proc. 38th Int'l Symp. on Computat. Geometry \textup{(SoCG 2022)}}, vol. 224 of \emph{LIPIcs}, pp. 38:1--38:14. Schloss Dagstuhl, 2022.

\bibitem[{Esperet et~al.(2023)Esperet, Joret, and Morin}]{EJM23}
\textsc{Louis Esperet, Gwena\"{e}l Joret, and Pat Morin}.
\newblock \href{https://doi.org/10.1112/jlms.12781}{Sparse universal graphs for planarity}.
\newblock \emph{J. London Math. Soc.}, 108(4):1333--1357, 2023.

\bibitem[{Frieze and Tsourakakis(2014)}]{FT14}
\textsc{Alan Frieze and Charalampos~E. Tsourakakis}.
\newblock \href{https://doi.org/10.1080/15427951.2013.796300}{Some properties of random {A}pollonian networks}.
\newblock \emph{Internet Math.}, 10(1-2):162--187, 2014.

\bibitem[{Gawrychowski and Janczewski(2022)}]{GJ22}
\textsc{Pawel Gawrychowski and Wojciech Janczewski}.
\newblock \href{https://doi.org/10.1137/1.9781611977066.3}{Simpler adjacency labeling for planar graphs with {B}-trees}.
\newblock In \textsc{Karl Bringmann and Timothy~M. Chan}, eds., \emph{Proc. 5th Symposium on Simplicity in Algorithms \textup{(SOSA@SODA 2022)}}, pp. 24--36. {SIAM}, 2022.

\bibitem[{Harvey and Wood(2017)}]{HW17}
\textsc{Daniel~J. Harvey and David~R. Wood}.
\newblock \href{https://doi.org/10.1002/jgt.22030}{Parameters tied to treewidth}.
\newblock \emph{J. Graph Theory}, 84(4):364--385, 2017.

\bibitem[{Huynh et~al.(2021)Huynh, Mohar, {\v{S}}{\'a}mal, Thomassen, and Wood}]{HMSTW}
\textsc{Tony Huynh, Bojan Mohar, Robert {\v{S}}{\'a}mal, Carsten Thomassen, and David~R. Wood}.
\newblock \href{https://arxiv.org/abs/2109.00327}{Universality in minor-closed graph classes}.
\newblock 2021, arXiv:2109.00327.

\bibitem[{Jacob and Pilipczuk(2022)}]{JP22}
\textsc{Hugo Jacob and Marcin Pilipczuk}.
\newblock \href{https://doi.org/10.1007/978-3-031-15914-5\_21}{Bounding twin-width for bounded-treewidth graphs, planar graphs, and bipartite graphs}.
\newblock In \textsc{Michael~A. Bekos and Michael Kaufmann}, eds., \emph{Proc. 48th International Workshop on Graph-Theoretic Concepts in Computer Science \textup{({WG} 2022})}, vol. 13453 of \emph{Lecture Notes in Comput. Sci.}, pp. 287--299. Springer, 2022.

\bibitem[{Johnson et~al.(2001)Johnson, Robertson, Seymour, and Thomas}]{JRST-JCTB01}
\textsc{Thor Johnson, Neil Robertson, Paul~D. Seymour, and Robin Thomas}.
\newblock \href{https://doi.org/10.1006/jctb.2000.2031}{Directed tree-width}.
\newblock \emph{J. Combin. Theory Ser. B}, 82(1):138--154, 2001.

\bibitem[{Knauer and Ueckerdt(2012)}]{KU12}
\textsc{Kolja Knauer and Torsten Ueckerdt}.
\newblock \href{https://kam.mff.cuni.cz/workshops/mcw/work18/mcw2012booklet.pdf}{Simple treewidth}.
\newblock In \textsc{Pavel Ryt\'ir}, ed., \emph{Midsummer Combinatorial Workshop Prague}. 2012.

\bibitem[{Kratochv{\'{\i}}l and Vaner(2012)}]{KV12}
\textsc{Jan Kratochv{\'{\i}}l and Michal Vaner}.
\newblock \href{http://arxiv.org/abs/1210.8113}{A note on planar partial 3-trees}.
\newblock arXiv:1210.8113, 2012.

\bibitem[{Kr\'{a}\v{l} et~al.(2024)Kr\'{a}\v{l}, Pek\'{a}rkov\'{a}, and \v{S}torgel}]{KPS24}
\textsc{Daniel Kr\'{a}\v{l}, Krist\'{y}na Pek\'{a}rkov\'{a}, and Kenny \v{S}torgel}.
\newblock \href{https://drops.dagstuhl.de/entities/document/10.4230/LIPIcs.MFCS.2024.66}{Twin-width of graphs on surfaces}.
\newblock In \textsc{Rastislav Kr\'{a}lovi\v{c} and Anton{\'\i}n Ku\v{c}era}, eds., \emph{Proc. 49th Int'l Symposium on Mathematical Foundations of Computer Science \textup{(MFCS 2024)}}, vol. 306 of \emph{LIPIcs}, pp. 66:1--66:15. Schloss Dagstuhl, 2024.

\bibitem[{Liu et~al.(2024)Liu, Norin, and Wood}]{LNW}
\textsc{Chun-Hung Liu, Sergey Norin, and David~R. Wood}.
\newblock \href{https://arxiv.org/abs/2410.20333}{Product structure and tree decompositions}.
\newblock 2024, arXiv:2410.20333.

\bibitem[{Markenzon et~al.(2006)Markenzon, Justel, and Paciornik}]{MJP06}
\textsc{Lilian Markenzon, Claudia~Marcela Justel, and N.~Paciornik}.
\newblock \href{https://doi.org/10.1016/j.dam.2005.05.021}{Subclasses of {$k$}-trees: characterization and recognition}.
\newblock \emph{Discrete Appl. Math.}, 154(5):818--825, 2006.

\bibitem[{Reed(1997)}]{Reed97}
\textsc{Bruce~A. Reed}.
\newblock \href{https://doi.org/10.1017/CBO9780511662119.006}{Tree width and tangles: a new connectivity measure and some applications}.
\newblock In \textsc{R.~A. Bailey}, ed., \emph{Surveys in Combinatorics}, vol. 241 of \emph{London Math. Soc. Lecture Note Ser.}, pp. 87--162. Cambridge Univ. Press, 1997.

\bibitem[{Wulf(2016)}]{Wulf16}
\textsc{Lasse Wulf}.
\newblock \href{https://i11www.iti.kit.edu/_media/teaching/theses/ba-wulf-16.pdf}{Stacked treewidth and the {C}olin de {V}erdi\'ere number}.
\newblock 2016.
\newblock Bachelorthesis, Institute of Theoretical Computer Science, Karlsruhe Institute of Technology.

\end{thebibliography}
    }

\appendix

\section{Proof of \cref{lem:dpui_ub}}
\customlabel{appendix:A}{Appendix A}

\citet{DOSV98} proved that:
\begin{lem}
    \label{lem:partitions}
    For any integers $p, q,k \geq 1$ with $k = p + q + 1$, every graph of treewidth $k$ has a vertex partition into two induced subgraphs with treewidth at most $p$ and $q$ respectively.
\end{lem}

Let \defn{$G^+$} be the graph obtained from $G$ by adding a dominant vertex adjacent to all existing vertices. The proof of \cref{lem:lnw} is similar to that given in \citep[Proposition 3.6]{LNW}.

\begin{lem}
    \label{lem:lnw}
    For any vertex partition $V_1, V_2$ of $V(G)$, one can orient the edges of $G[V_1]^+$ and $G[V_2]^+$ so that $G \subsetsim \overrightarrow{G[V_1]^+} \, \squareslash \, \overrightarrow{G[V_2]^+}$.
\end{lem}
\begin{proof}
    For each $i \in \{1, 2\}$, let the dominant vertex in $G[V_i]^+$ be $r_i$. Map each vertex $x \in V_1$ to $(x,r_2) \in V(G[V_1]^+) \times V(G[V_2]^+)$ and map each vertex $y \in V_2$ to $(r_1,y) \in V(G[V_1]^+) \times V(G[V_2]^+)$. For the graph $G[V_1]^+$, direct edges incident to $r_1$ away from $r_1$ and the rest of the edges arbitrarily. For the graph $G[V_2]^+$, direct edges incident to $r_2$ towards $r_2$ and the rest of the edges arbitrarily. Observe that for any $x \in V_1$ and $y \in V_2$, the edge $(r_1,y) (x, r_2)$ is in the directed product. This completes the proof.
\end{proof}

\restatabledpuiub*

\begin{proof}
    By \cref{lem:partitions}, one can partition $G$ into induced subgraphs $H_1'$ and $H_2'$ of treewidth at most $p'$ and $q'$ respectively, where $p' = p - 1$ and $q' = q - 1$. So $p' + q' \geq k -1$. Let $H_1:= H_1^+$ and $H_2:= H_2^+$, so $\tw(H_1) \leq p$ and $\tw(H_2) \leq q$. \cref{lem:lnw} implies  there is an orientation of the edges of $H_1$ and $H_2$ so that $G \subsetsim \vec H_1 \, \squareslash \, \vec H_2$.
\end{proof}

\section{Simple Treewidth Graphs}
\customlabel{appendix:B}{Appendix B}
This appendix presents a proof of \cref{lem:edge_max_stw_k}.

We say that $H$ is a \defn{spanning subgraph} of $G$ if $H \subseteq G$ and $V(H) = V(G)$. Let $G$ be a graph. A tree-decomposition $(T, \mathcal{B})$ of $G$ is \defn{normal} if for any two distinct bags $B_x, B_y \in \mathcal{B}$, we have $B_x \not\subseteq B_y$ and $B_y \not\subseteq B_x$.
For an integer $k \geq 1$, a tree-decomposition $(T, \mathcal{B})$ is \defn{$k$-smooth} if each bag $B_x \in \mathcal{B}$ has size $k+1$, and for each $xy \in E(T)$, $\left|B_x \cap B_y\right| = k$. If $S \subseteq V(G)$, then let $\mathdefn{\Phi_{(T, \mathcal{B})}(S)}:= \{ x \in V(T) : S \subseteq B_x \}$ denote the nodes whose bags contain $S$. Recall that a tree-decomposition $(T, \mathcal{B})$ of a graph $G$ is \defn{$k$-simple} if it has width at most $k$, and for each set $S \subseteq V(G)$ of $k$ vertices, $|\Phi_{(T, \mathcal{B})}(S)|\leq 2$.

The main goal of this section is to prove \cref{thm:simple_treewidth}, which strengthens \cref{lem:edge_max_stw_k}. Note that \cref{lem:edge_max_stw_k} is trivial if $\left|V(G)\right| \leq k$.
\begin{restatable}{thm}{thmstw}
    \label{thm:simple_treewidth}
    Let $k \geq 1$ be an integer and $G$ be a graph with $|V(G)| \geq k+1$. The following are equivalent:
    \begin{enumerate}[label=(\arabic*):,leftmargin=3em]
        \item $\stw(G) \leq k$; that is,  $G$ has a $k$-simple tree-decomposition.
        \item $G$ has a normal $k$-simple $k$-smooth tree-decomposition.
        \item $G$ is a spanning subgraph of a simple $k$-tree.
    \end{enumerate}
\end{restatable}
\vspace{-0.5em}
\begin{lem}\label{lem:connected}
    Let $(T, \mathcal{B})$ be a tree-decomposition of a graph $G$. Then for all $S \subseteq V(G)$, the subtree of $T$ induced by $\Phi_{(T, \mathcal{B})}(S)$ is connected.
\end{lem}
\begin{proof}
    % The lemma holds if $\left|S\right| = 1$ by the definition of a tree-decomposition, so assume $\left|S\right| \geq 2$.
    Let $B_x$ and $B_y$ be distinct bags in $\mathcal{B}$ containing $S$. For each vertex $v \in S$, the subtree induced by $\Phi_{(T, \mathcal{B})}(\{v\})$ contains the unique $xy$-path $P$ in $T$. Then for each node $x \in V(P)$, we have $S \subseteq B_x$. Thus the subtree of $T$ induced by $\Phi_{(T, \mathcal{B})}(S)$ is connected.
\end{proof}
We show that $(T, \mathcal{B})$ is normal if and only if for each edge $xy \in E(T)$,  we have $B_x \not\subseteq B_y$ and $B_y \not\subseteq B_x$. Suppose for each edge $ab \in E(T)$,  we have $B_a \not\subseteq B_b$ and $B_b \not\subseteq B_a$, but $(T, \mathcal{B})$ is not normal. Then there are distinct bags $B_x$ and $B_y$ such that $B_x \subseteq B_y$. By \cref{lem:connected}, the subtree of $T$ induced by $\Phi_{(T, \mathcal{B})}(B_x)$ is connected, and contains the unique $xy$-path $P$ in $T$. Let $z \in V(P)$ be the vertex adjacent to $x$. Then $B_x \subseteq B_z$, a contradiction.

Let $r \in V(T)$. A \defn{tree-decomposition rooted at $r$}, denoted \defn{$(T, \mathcal{B}, r)$}, is a tree-decomposition $(T, \mathcal{B})$, where $T$ is rooted at $r$. A \defn{rooted tree-decomposition} is a tree-decomposition rooted at $r$ for some $r \in V(T)$.
The \defn{depth} of a vertex $x \in V(T)$, denoted by $\mathdefn{\depth(x)}$, is the distance from $x$ to $r$ in $T$. If $xy \in E(T)$ and $\depth(x) > \depth(y)$, then we say $y$ is the \defn{parent} of $x$ and $x$ is a \defn{child} of $y$. For $x \in V(T)$, define \defn{$T_x$} to be the subtree of $T$ induced by vertices $v$ whose $rv$-path in $T$ passes through $x$. Note that $x \in V(T_x)$.

For an integer $k \geq 1$, a rooted tree-decomposition $(T, \mathcal{B}, r)$ of $G$ is \defn{$k$-fine} if $|B_r| = k+1$.
Let \defn{$\mathcal{T}_k(G)$} denote the set of all rooted normal $k$-fine $k$-simple tree-decompositions of $G$. Note that $\mathcal{T}_k(G)$ may be empty, for instance if $\stw(G) \geq k+1$ or if $|V(G)| \leq k$.
Lastly, the \defn{score} of a rooted tree-decomposition $(T, \mathcal{B}, r)$ is:
\begin{equation*}
    \mathdefn{\score((T, \mathcal{B}, r))} := \sum_{x \in V(T)} (\depth(x) + 1) |B_x|.
\end{equation*}

\begin{lem}\label{lem:well_def}
    Let $k \geq 1$ be an integer. Suppose $G$ is a graph with $\left|V(G)\right| \geq k+1$ and $\stw(G) \leq k$. Then $\mathcal{T}_k(G)$ is non-empty, and the score of each rooted tree-decomposition $(T, \mathcal{B}, r) \in \mathcal{T}_k(G)$ is at most $(k+1) \cdot |V(G)| \cdot (|V(G)| + 1)$.
\end{lem}
\begin{proof}
    We first prove that $\mathcal{T}_k(G)$ is non-empty. Let $(T, \mathcal{B})$ be a $k$-simple tree-decomposition of $G$ such that:
    \begin{enumerate}[label=\textbf{(C\arabic*):},leftmargin=4em]
        \item $|\mathcal{B}|$ is minimised;
        \item subject to (C1), the width of $(T,\mathcal{B})$ is maximised.
    \end{enumerate}
    If $(T, \mathcal{B})$ is not normal, then there is an edge $xy \in E(T)$ such that $B_x \subseteq B_y$. We obtain a new tree-decomposition with a smaller number of bags by contracting the edge $xy$ into $y$ and discarding the bag $B_x$, which contradicts (C1). Thus, $(T,\mathcal{B})$ is normal.

    Let $B_x \in \mathcal{B}$ be a bag of maximum size and suppose $\left|B_x\right| \leq k$. Then each bag in $\mathcal{B}$ has size at most $k$. Since $|V(G)| \geq k+1$, we deduce that $|V(T)| \neq 1$, and therefore there is a neighbour $y$ of $x$ in $T$. Since $(T, \mathcal{B})$ is normal, there exists $v \in B_y \setminus B_x$. Construct a $T$-decomposition $\mathcal{B}'$ by adding $v$ to $B_x$ and leaving all other bags unmodified. If $(T, \mathcal{B}')$ is $k$-simple then it contradicts (C2), so $(T, \mathcal{B}')$ is not $k$-simple. Therefore, there is a set $S \subseteq V(G)$ of $k$ vertices contained in three bags of $\mathcal{B}'$. Two of these bags are different from $x$; let $y, z \in V(T)$ be distinct nodes such that $x \neq z, x \neq y$, $S \subseteq B_y$ and $S \subseteq B_z$. However, $|B_y|, |B_z| \leq k$, so $B_y = B_z = S$. This contradicts the normality of $(T, \mathcal{B})$. Thus, $\left|B_x\right| = k+1$. The rooted tree-decomposition $(T, \mathcal{B}, x)$ is $k$-fine, and therefore $(T, \mathcal{B}, x) \in \mathcal{T}_k(G)$.

    For the score bound, observe that the number of bags of any normal tree-decomposition of $G$ is at most $|V(G)|$. If $(T, \mathcal{B}, r) \in \mathcal{T}_k(G)$, then
    \begin{equation*}
        \score((T, \mathcal{B}, r))\leq \sum_{x \in V(T)} (k+1) (|V(G)| + 1) \leq (k+1) \cdot |V(G)| \cdot (|V(G)| + 1),
    \end{equation*}
    as desired.
\end{proof}
Observe that \cref{lem:well_def} implies that there is a rooted tree-decomposition $(T, \mathcal{B}, r)$ in $\mathcal{T}_k(G)$ maximising the score.
\begin{lem}
    \label{lem:rooted_score}
    Let $k \geq 1$ be an integer. Suppose $G$ is a graph with $\left|V(G)\right| \geq k+1$ and $\stw(G) \leq k$. Let $(T, \mathcal{B}, r)$ be a rooted tree-decomposition in $\mathcal{T}_k(G)$ maximising the score. Let $y \in V(T)$, $x$ be a child of $y$, and let $S := B_x \cap B_y$. Then no child $z \neq x$ of $y$ has $S \subseteq B_z$.
\end{lem}
\begin{proof}
    Suppose $z \neq x$ is a child of $y$ and $S \subseteq B_{z}$. Consider the tree $T'$ obtained from $T$ by removing the edge $xy$ and adding the edge $xz$. We claim that $(T', \mathcal{B})$ is a tree-decomposition of $G$.
    It suffices to show that for each vertex $v \in V(G)$, $\Phi_{(T', \mathcal{B})}(\{v\})$ induces a connected subtree of $T'$. This holds because $B_x \cap B_y \subseteq B_z$.
    Since $\depth(z) > \depth(y)$, the depth of each node in $T_x$ increases by $1$. Since $\mathcal{B}$ remains unmodified, we have $(T', \mathcal{B}, r) \in \mathcal{T}_k(G)$, and $(T', \mathcal{B}, r)$ has a higher score than $(T, \mathcal{B}, r)$, a contradiction.
\end{proof}
\begin{lem}
    \label{lem:sizes}
    Let $k \geq 1$ be an integer. Suppose $G$ is a graph with $\left|V(G)\right| \geq k+1$ and $\stw(G) \leq k$. Let $(T, \mathcal{B}, r)$ be a rooted tree-decomposition in $\mathcal{T}_k(G)$ maximising the score. Then each bag of $\mathcal{B}$ has size $k+1$.
\end{lem}
\begin{proof}
    Suppose there is a node $x \in V(T)$ such that $|B_x| \leq k$. Let $x$ be the node closest to $r$. Since $\left|B_r\right| = k+1$, $x \neq r$. Let $y$ be the parent of $x$, so $\left|B_y\right| = k+1$. If $\left|B_y \setminus B_x\right| = 1$, then $B_x \subseteq B_y$, contradicting the assumption that $(T, \mathcal{B})$ is normal. Therefore, $\left|B_y \setminus B_x\right| \geq 2$.

    Suppose $\mathcal{B}= (B_t : t \in V(T))$. For each vertex $v \in B_y \setminus B_x$, we construct a $T$-decomposition of $G$ denoted by $\mathcal{B}_v := (B'_t : t \in V(T))$, where $B_x':= (B_x + v)$, and all other bags are unmodified.
    We claim that $(T, \mathcal{B}_v)$ is a normal tree-decomposition of $G$.  Since all bags apart from $B_x$ are unmodified, it suffices to check that for each neighbour $z$ of $x$ in $T$, $B'_{x} \not\subseteq B_z$ and $B_z \not\subseteq B'_{x}$. Since $B_x \not\subseteq B_z$, $B'_{x} \not\subseteq B_z$.
    To show that $B_z \not\subseteq B_x'$, consider two cases. If $z \neq y$, then since $(T, \mathcal{B})$ is normal, we have $B_z \not\subseteq B_x$, and there is a vertex $w \in B_z \setminus B_x$. By \cref{lem:connected} and the fact that $v \notin B_x$, $v \neq w$. Therefore, $w \in B_z \setminus (B_x + v)$ and $B_z \not\subseteq B'_{x}$. If $z = y$, then since $\left|B_y \setminus B_x\right| \geq 2$, there is a vertex $w\in B_y \setminus B_x$ with $w \neq v$, so $B_y \not\subseteq B'_{x}$. So $(T, \mathcal{B}_v)$ is a normal tree-decomposition for each vertex $v \in B_y \setminus B_x$.

    If $(T, \mathcal{B}_v)$ is $k$-simple, then $(T, \mathcal{B}_v, r)$ is in $\mathcal{T}_k(G)$, and $(T, \mathcal{B}_v, r)$ is a rooted tree-decomposition that attains a higher score than $(T, \mathcal{B}, r)$, a contradiction. Therefore, we deduce that for each $v \in B_y \setminus B_x$, $(T, \mathcal{B}_v)$ is not $k$-simple, and there is a subset $S'_v \subseteq B_x$ of $k-1$ vertices such that $S'_v + v$ is contained in three bags of $\mathcal{B}_v$; that is,  $|\Phi_{(T, \mathcal{B}_v)}(S'_v + v)| = 3$. Since $S'_v + v \subseteq B_x'$, $x \in \Phi_{(T, \mathcal{B}_v)}(S'_v + v)$. Since $v \notin B_x$, by \cref{lem:connected}, $x$ is the only vertex in $V(T_x)$ in $\Phi_{(T, \mathcal{B}_v)}(S'_v + v)$. Therefore, by \cref{lem:connected}, $y \in \Phi_{(T, \mathcal{B}_v)}(S'_v + v)$.
    Therefore, $S'_v \subseteq B_y$, and $S'_v \subseteq B_x \cap B_y$.  Since $\left| S'_v\right| = k-1$, $|B_x \cap B_y| \geq k-1$. However, $\left|B_y \setminus B_x\right| \geq 2$, so $\left|B_x \cap B_y\right| = k-1$ and $B_x \cap B_y = S'_v$. All $S_v'$ are equal, and we define $S':= B_x \cap B_y$. Since $(T, \mathcal{B})$ is normal, we have $B_x \setminus B_y \neq \varnothing$, and $\left|B_x\right| = k$.

    Let $v_1, v_2 \in B_y \setminus B_x$ be distinct. For each $i \in \{1, 2\}$, because $|\Phi_{(T, \mathcal{B}_v)}(S'_v + v)| = 3$, we let $\Phi_{(T, \mathcal{B}_v)}(S'_v + v) = \{x, y, z_i\}$. By \cref{lem:connected}, $z_1$ and $z_2$ are adjacent to $y$. If $z_1 = z_2$, then since $v_1, v_2 \in B_{z_1}$ and $S \subseteq B_{z_1}$, we have $B_{z_1} = B_{y}$, contradicting the normality of $(T, \mathcal{B})$. Hence, $z_1 \neq z_2$. Since in a rooted tree each non-root vertex has a unique parent, there is an index $i \in \{ 1, 2 \}$ such that $\depth(z_i) > \depth(y)$. However, then $z_i \neq x$ is a child of $y$ with $B_x \cap B_y \subseteq B_{z_i}$, contradicting \cref{lem:rooted_score}.
\end{proof}

\begin{lem}
    \label{lem:symmetric_diff}
    Let $k \geq 1$ be an integer. Suppose $G$ is a graph with $\left|V(G)\right| \geq k+1$ and $\stw(G) \leq k$. Let $(T, \mathcal{B}, r)$ be a rooted tree-decomposition in $\mathcal{T}_k(G)$ maximising the score. Then for each edge $xy \in E(T)$, we have $\left|B_y \setminus B_x\right| = 1$ and $|B_x \setminus B_y| = 1$.
\end{lem}
\begin{proof}
    Without loss of generality, suppose $\depth(x) > \depth(y)$. Suppose for contradiction that $\left|B_y \setminus B_x\right| \geq 2$. By \cref{lem:sizes}, $\left|B_x\right| = \left|B_y\right| = k+1$. Therefore, $|B_x \cap B_y| \leq k-1$ and $|B_x \setminus B_y| \geq 2$.

    Let $T'$ be the tree obtained from $T$ by replacing the edge $xy$ with a $2$-edge path $xzy$. Suppose $\mathcal{B}= (B_t : t \in V(T))$. For each $u \in B_x \setminus B_y$ and $v \in B_y \setminus B_x$, we construct a $T'$-decomposition of $G$ given by $\mathcal{B}_{uv}:= ( B'_t: t \in V(T') )$, where $B_z' := ((B_x \cap B_y) + u + v)$, and all other bags are unmodified.

    We claim that $(T', \mathcal{B}_{uv})$ is a normal tree-decomposition of $G$ with width at most $k$. To see that $(T', \mathcal{B}_{uv})$ is a tree-decomposition, we check that for each vertex $g \in V(G)$, the subtree of $T'$ induced by $\Phi_{(T', \mathcal{B}_{uv})}(\{g\})$ is connected. This is because if $g  \in B_x$ and $g \in B_y$, then $g \in B_z'$. The other properties of $(T', \mathcal{B}_{uv})$ being a tree-decomposition are straightforward to verify. Since $\left|B_x \setminus B_y\right| \geq 2$ and $\left|B_y \setminus B_x\right| \geq 2$, $B_x \not \subseteq B_z'$ and $B_y \not \subseteq B_z'$. Since $v \notin B_x$ and $u \notin B_y$, we have $B_z' \not \subseteq B_x$ and $B_z' \not \subseteq B_y$, hence $(T', \mathcal{B}_{uv})$ is normal. $(T', \mathcal{B}_{uv})$ has width at most $k$ because $|B_z'| = \left|B_x \cap B_y\right| + 2 \leq k+1$.

    Observe that the depth of each vertex in the subtree rooted at $x$ increases by $1$ in $T'$. If $(T', \mathcal{B}_{uv})$ is $k$-simple, then $(T', \mathcal{B}_{uv}, r)$ is in $\mathcal{T}_k(G)$ and has a higher score than $(T, \mathcal{B}, r)$, a contradiction. Thus, $(T', \mathcal{B}_{uv})$ is not $k$-simple, and there is a set $S \subseteq V(G)$ of size $k$ such that $|\Phi_{(T', \mathcal{B}_{uv})}(S)| = 3$. If $\left|B_x \cap B_y\right| < k-1$, then $|B_z'| \leq k$ and we deduce that $S = B_z'$ and $u, v \in B_z'$. However, $S \not \subseteq B_x$ and $S \not \subseteq B_y$ since $u \notin  B_y$ and $v \notin B_x$. Then $|\Phi_{(T', \mathcal{B}_{uv})}(S)| = 3$ contradicts \cref{lem:connected}. Therefore, $\left|B_x \cap B_y\right| = k-1$, $|B_x| = |B_y| = k+1$, and $\left|B_y \setminus B_x\right| = \left|B_x \setminus B_y\right| = 2$.

    We claim that there is a vertex $v \in B_y \setminus B_x$ such that if $S:= (B_x \cap B_y) + v$, then $|\Phi_{(T, \mathcal{B})}(S)| \leq 1$.
    Since $\left|B_y \setminus B_x\right| = 2$, there are distinct vertices $v_1, v_2 \in B_y \setminus B_x$. For each $i \in \{ 1, 2 \}$, let $S_i= (B_x \cap B_y) + v_i$. Suppose that $|\Phi_{(T, \mathcal{B})}(S_i)| \geq 2$ for each $i \in \{1, 2\}$. By \cref{lem:connected}, there exist nodes $z_1, z_2 \in V(T)$ adjacent to $y$ such that $S_1 \subseteq B_{z_1}$ and $S_2 \subseteq B_{z_2}$. If $z_1 = z_2$, then $v_1, v_2 \in B_{z_1}$, and $B_{z_1} = (B_x \cap B_y) + v_1 + v_2 = B_y$, a contradiction. Further, because $v_1 \notin B_x$ and $v_2 \notin B_x$, then $z_1 \neq x$ and $z_2 \neq x$. Hence, there exists $i \in \{ 1, 2 \}$ such that $\depth(z_i) > \depth(y)$. However, then $B_x \cap B_y \subseteq B_{z_i}$, contradicting \cref{lem:rooted_score}. Therefore, there is a vertex $v \in B_y \setminus B_x$ such that $|\Phi_{(T, \mathcal{B})}((B_x \cap B_y) + v)| \leq 1$.

    We claim that there is a vertex $u \in B_x \setminus B_y$ such that if $S:= (B_x \cap B_y) + u$, then $|\Phi_{(T, \mathcal{B})}(S)| \leq 1$.
    Since $\left|B_x \setminus B_y\right| = 2$, there are distinct vertices $u_1, u_2 \in B_x \setminus B_y$. For each $i \in \{ 1, 2 \}$, let $S_i= (B_x \cap B_y) + u_i$. Suppose that $|\Phi_{(T, \mathcal{B})}(S_i)| \geq 2$ for each $i \in \{1, 2\}$. By \cref{lem:connected}, there exist nodes $z_1, z_2$ adjacent to $x$ such that $S_1 \subseteq B_{z_1}$ and $S_2 \subseteq B_{z_2}$. If $z_1 = z_2$, then $u_1, u_2 \in B_{z_1}$ and $B_{z_1} = (B_x \cap B_y) + u_1 + u_2 = B_x$, a contradiction. Further, because $u_1 \notin B_y$ and $u_2 \notin B_y$, then $z_1 \neq y$ and $z_2 \neq y$. Therefore, for each $i \in \{ 1, 2 \}$, $\depth(z_i) > \depth(x)$. Let $T_i$ be the subtree rooted at $z_i$. Without loss of generality, assume that $\left|V(T_1)\right| \geq \left|V(T_2)\right|$.

    Consider the tree $T^*$ obtained from $T$ by removing the edge $xy$ and adding the edge $yz_2$. To see that $(T^*, \mathcal{B})$ is a valid tree-decomposition of $G$, we show that for each vertex $g \in V(G)$, the subtree of $T^*$ induced by $\Phi_{(T^*, \mathcal{B})}(\{g\})$ is connected. It suffices to show that if $v \in B_x$ and $v \in B_{y}$, then $v \in B_{z_2}$, which is true since $B_x \cap B_y \subseteq B_{z_2}$.

    The depth of each node in $V(T_2)$ decreases by $1$, while the depth of $x$ and each node in $V(T_1)$ increases by $1$. Consequently,
    \begin{align*}
        \score((T^*, \mathcal{B}, r)) & \geq \score((T, \mathcal{B}, r)) + (|V(T_1)| + 1)(k+1) - |V(T_2)|(k+1).
    \end{align*}
    Observe that $(T^*, \mathcal{B}, r)$ is normal, $k$-fine and $k$-simple, so $(T^*, \mathcal{B}, r)$ is in $\mathcal{T}_k(G)$.
    Since $\left|V(T_1)\right| \geq \left|V(T_2)\right|$, $(T^*, \mathcal{B}, r)$ attains a higher score than $(T, \mathcal{B}, r)$, a contradiction. Therefore, there is a vertex $u \in B_x \setminus B_y$ such that $|\Phi_{(T, \mathcal{B})}((B_x \cap B_y) + u)| \leq 1$. However, then $(T', \mathcal{B}_{uv})$ is $k$-simple, which is a contradiction since we have previously shown that $(T', \mathcal{B}_{uv})$ is not $k$-simple. This completes the proof.
\end{proof}

\thmstw*

\begin{proof}
    (1) $\Rightarrow$ (2): Let $(T, \mathcal{B}, r)$ be a rooted tree-decomposition in $\mathcal{T}_k(G)$ maximising the score. By \cref{lem:sizes,lem:symmetric_diff}, $(T, \mathcal{B})$ satisfies (2).

    (2) $\Rightarrow$ (3): Let $(T, \mathcal{B})$ be a normal $k$-simple $k$-smooth tree-decomposition of $G$. Root $T$ at a node $t_1 \in V(T)$. Suppose $B_{t_1} = \{v_1, \dots, v_{k+1} \}$, and let $t_1 \preceq \dots \preceq t_{|V(T)|}$ be a BFS-ordering of $V(T)$. For a non-root node $t_i \in V(T)$, let \defn{$p(t_i)$} denote the parent of $t_i$ in $T$. Since $(T, \mathcal{B})$ is $k$-smooth, for each non-root node $t_i \in V(T)$, there is a vertex $v_{t_i}$ such that $\{v_{t_i}\} = B_{t_i} \setminus B_{p(t_i)}$. Define $G'$ by adding edges to $G$ so that each bag is a $(k+1)$-clique. Note that $G$ is a spanning subgraph of $G'$.

    We claim that $G'$ is a $k$-tree. $G'$ can be constructed by  starting from a $(k+1)$-clique on the vertices $\{v_1, \dots, v_{k+1}\}$. In the order of the BFS-ordering, for each non-root node $t_i \in V(T)$ ($i \geq 2$), we recursively add the vertex $v_{t_i} \in B_{t_i} \setminus B_{p(t_i)}$ to $G'$ so that $v_{t_i}$ is adjacent to all vertices in the $k$-clique $B_{t_i} \cap B_{p(t_i)}$. For each non-root node $t_i \in V(T)$, we have $p(t_i) \prec t_i$, so the construction is well-defined.

    We claim that $G'$ is a simple $k$-tree. Suppose to the contrary. Then there is a $k$-clique $C$ used at least twice in the construction of $G'$. Suppose $C$ is used at steps $i$ and $j$ for some $i \neq j$. Note that $i, j \geq 2$. By construction of $G'$, $B_{t_i} = C + v_{t_i}$, $B_{t_j} = C + v_{t_j}$ and note that $\{v_{t_i}\} = B_{t_i} \setminus B_{p(t_i)}$ and $\{v_{t_j}\} = B_{t_j} \setminus B_{p(t_j)}$. Since $(T, \mathcal{B})$ is $k$-smooth, $|B_{t_i} \cap B_{p(t_i)}| = |B_{t_j} \cap B_{p(t_j)}| = k$, so $C = B_{t_i} \cap B_{p(t_i)} = B_{t_j} \cap B_{p(t_j)}$. Consequently, $C \subseteq V(G)$ is a set of $k$ vertices contained in $B_{t_i}, B_{t_j}, B_{p(t_i)}, B_{p(t_j)} \in \mathcal{B}$. Since $i \neq j$, at least three of these bags are distinct, contradicting the assumption that $(T, \mathcal{B})$ is $k$-simple.

    (3) $\Rightarrow$ (1): It suffices to show that every simple $k$-tree $G$ has a $k$-simple tree-decomposition. Suppose $G$ is constructed from a $(k+1)$-clique on the vertices $\{v_1, \dots, v_{k+1}\}$, and then by adding each of the remaining vertices $v_{k+2}, \dots, v_{|V(G)|}$ in order.
    Let the simple $k$-tree obtained at the $i$-th step of this process be $G_i$. We recursively construct a tree-decomposition $(T_i, \mathcal{B}_i)$ of $G_i$. When $i = 1$, $T_1$ is the tree defined by $V(T_1) = \{t_1\}$ and $E(T_1) = \varnothing$. Let $\mathcal{B}_1$ be the $T_1$-decomposition of $G_1$ defined by $B_{t_1} := \{v_1, \dots, v_{k+1}\}$. For $i \geq 2$, suppose $v_{k+i} \in V(G_i) \setminus V(G_{i-1})$ has neighbourhood $C$, where $C$ is a $k$-clique in $V(G_{i-1})$. Let $y \in V(T_{i-1})$ be a node such that $C \subseteq B_y$. We construct $T_i$ from $T_{i-1}$ by adding a new node $t_i$ adjacent to $y$, and we define $\mathcal{B}_i$ to be the $T_i$-decomposition of $G_i$ obtained by adding the bag $B_{t_i}:= C + v_{k+i}$ to $\mathcal{B}_{i-1}$.
    Let $(T, \mathcal{B})$ be the tree-decomposition of $G$ obtained at the end of this process, and root $T$ at the node $t_1$. By construction, $(T, \mathcal{B})$ is $k$-smooth, and for each non-root node $t_i \in V(T)$, the vertex $v_{k+i} \in B_{t_i} \setminus B_{p(t_i)}$ is adjacent to each vertex in $B_{t_i} \cap B_{p(t_i)}$ in $G$.

    Suppose $(T, \mathcal{B})$ is not $k$-simple. Then there is a set $S \subseteq V(G)$ of $k$ vertices contained in three bags of $\mathcal{B}$. Since each bag $B_{t_i} \in \mathcal{B}$ induces a $(k+1)$-clique in $G$, $S$ is a $k$-clique in $G$. Let $t_{a}, t_{b}, t_{c} \in V(T)$ be distinct nodes such that $S \subseteq B_{t_a}$, $S \subseteq B_{t_b}$ and $S \subseteq B_{t_c}$ (so $a, b, c \geq 1$ are distinct integers). By \cref{lem:connected}, we may without loss of generality assume that $t_a, t_b$ and $t_c$ induce a $3$-vertex-path in $T$. We may further assume either I) $t_a = p(t_b)$ and $t_a = p(t_c)$ or II) $t_a = p(t_b)$ and $t_b = p(t_c)$. In either case, $B_{p(t_b)}$ and $B_{p(t_c)}$ contain $S$, and the $k$-clique $S$ is used at least twice (by $v_{k+b} \in B_{t_b} \setminus B_{p(t_b)}$ and $v_{k+c} \in B_{t_c} \setminus B_{p(t_c)}$) in the construction of $G$, a contradiction.
\end{proof}
\end{document}